\documentclass{article}
\usepackage[paperwidth=7in, paperheight=10in, margin=.875in]{geometry}
 \usepackage[backref,colorlinks,linkcolor=red,anchorcolor=green,citecolor=blue]{hyperref}
\usepackage{amsfonts,amssymb}
\usepackage{amsmath, amsthm}
\usepackage{graphicx}
\usepackage{cite}
\usepackage{enumerate}
\newtheorem{theorem}{Theorem}[subsection]
\newtheorem{lemma}[theorem]{Lemma}
\newtheorem{remark}{Remark}
\newtheorem{cor}[theorem]{Corollary}

 \newtheorem{definition}{Definition}[subsection]

   \allowdisplaybreaks
\begin{document}
 \title{A consistent kinetic model for a two-component mixture of polyatomic molecules}


          \author{Christian Klingenberg, Marlies Pirner, \thanks{Dept. of Mathematics at W\"uerzburg University, Emil-Fischer- Str. 40, W\"uerzburg, 97074, Germany (klingen@mathematik.uni-wuerzburg.de, marlies.pirner@mathematik.uni-wuerzburg.de). https://www.mathematik.uni-wuerzburg.de/$\sim$ klingen/Welcome.html}
          \and Gabriella Puppo \thanks{Universit\'a degli studi dell' Insubria, Via Valleggio, Como, 22100, Italy, (gabriella.puppo@uninsubria.it).}}

         \pagestyle{myheadings} \markboth{A TWO- SPECIES MODEL FOR POLYATOMIC MOLECULES}{CHRISTIAN KLINGENBERG, MARLIES PIRNER, GABRIELLA PUPPO}
         \date{}
          \maketitle

          \begin{abstract}
We consider a multi component gas mixture with translational and internal energy degrees of freedom assuming that the number of particles of each species remains constant. We will illustrate the derived model  in the case of two species, but the model can be easily generalized to multiple species. The two species are allowed to have different degrees of freedom in internal energy and are modelled by a system of kinetic ES-BGK equations featuring two interaction terms to account for momentum and energy transfer between the species. 
We prove consistency of our model: conservation properties, positivity of the temperature,  H-theorem and convergence to a global equilibrium in the form of a global Maxwell distribution. Thus, we are able to derive the usual macroscopic conservation laws. For numerical purposes we apply the Chu reduction to the developed model for polyatomic gases and give an application for a gas consisting of a mono atomic and a diatomic species. 
          \end{abstract}


\section{Introduction}
 
 In this paper we shall concern ourselves with a kinetic description of gas mixtures for polyatomic molecules. In the case of mono atomic molecules and two species this is traditionally done via the Boltzmann equation for the density distributions $f_1$ and $f_2$, see for example \cite{Cercignani, Cercignani_1975}. Under certain assumptions the complicated interaction terms of the Boltzmann equation can be simplified by a so called BGK approximation, consisting of a collision frequency multiplied by the deviation of the distributions from local Maxwellians. This approximation should be constructed in a way such that it  has the same main properties of the Boltzmann equation namely conservation of mass, momentum and energy, further it should have an H-theorem with its entropy inequality and the equilibrium must still be Maxwellian.  BGK  models give rise to efficient numerical computations, which are asymptotic preserving, that is they remain efficient even approaching the hydrodynamic regime \cite{Puppo_2007, Jin_2010,Dimarco_2014, Bennoune_2008,  Bernard_2015, Crestetto_2012}. Evolution of a polyatomic gas is very important in applications, for instance air consists of a gas mixture of polyatomic molecules. But, most kinetic models modelling air deal with the case of a mono atomic  gas consisting of only one species.
 
  In the literature one can find two types of models for polyatomic molecules \textcolor{black}{using the classical description of  physics. We do not take into account quantum mechanical effects, for this see for example \cite{Orlach}. In classical thermodynamics,}  there are models which contain a sum of collision terms on the right-hand side corresponding to the elastic and inelastic collisions. Examples are the models of Rykov \cite{Rykov}, Holway \cite{Holway} and Morse \cite{Morse}. The other type of models contain only one collision term on the right-hand side taking into account both elastic and inelastic interactions. Examples for this are Bernard, Iollo, Puppo \cite{Bernard}  or the model by Bisi and Caceres \cite{Bisi} modelling chemical interactions.  In this paper we want to extend the model of Bernard, Iollo and Puppo \cite{Bernard} from one species of molecules to a gas mixture of polyatomic molecules. In contrast to mono atomic molecules, in a polyatomic gas energy is not entirely stored in the kinetic energy of its molecules but also in their rotational and vibrational modes. For simplification we present the model in the case of two species.\textcolor{black}{ We do not consider chemical reactions. For models which include chemical reactions, see for example \cite{Aliat}.}  We allow the two species to have different degrees of freedom in internal energy. For example, we may consider a mixture consisting of a mono atomic and a diatomic gas. In addition, we want to model it via an ES-BGK approach  in order to reproduce the correct Boltzmann hydrodynamic regime close to the asymptotic continuum limit. \textcolor{black}{The presence of a tensor in the attractors should allow to overcome the well known problem of incorrect Prandtl number (analogously to paper \cite{Perthame} for a single gas) but the proof is still lacking. At least, we will propose a model which is consistent in the special case of a mono-atomic single gas.} The ES-BGK approximation was suggested  by Holway in the case of one species \cite{Holway}. The H-Theorem of this model then was proven in \cite{Perthame}.  Brull and Schneider   relate this model to a minimization problem in \cite{brullschneider}.

The outline of the paper is as follows: in section \ref{sec1} we will present the extension of the BGK model for polyatomic molecules from \cite{Bernard} to two species of polyatomic molecules.  In section \ref{sec2}, we extend it to an ES-BGK model and check if it is well-defined. In sections \ref{sec3} to \ref{sec6} we prove the conservation properties and the H-theorem. We show the positivity of all temperatures and quantify the structure of the equilibrium. In section \ref{sec7}, we compare our model with an other model presented in the literature from \cite{Perthame} which considers an ES-BGK model for one species of polyatomic molecules. 
In section \ref{sec8} we apply the method of Chu reduction to our model in order to reduce the complexity of the variables for the rotational and vibrational energy degrees of freedom for numerical purposes.
In section \ref{sec9} we give an application in the case of a mono atomic and a diatomic molecule.



\section{The BGK approximation}
\label{sec1}
\textcolor{black}{In this section we first want to motivate how our model with several coupled equations will look like. For the convenience of the reader, we will summarize all this equations again at the end of the section such that one sees the whole model at a glance.}
For simplicity in the following we consider a mixture composed of two different species. Let $x\in \mathbb{R}^d$ and $v\in \mathbb{R}^d, d \in \mathbb{N}$ be the phase space variables  and $t\geq 0$ the time. Let $M$ be the total number of different rotational and vibrational degrees of freedom and $l_k$ the number of internal degrees of freedom of species $k$, $k=1,2$. Note that the sum $l_1+l_2$ is not necessarily equal to $M$, because the two species could both have the same internal degree of freedom. Then $\eta \in \mathbb{R}^{M}$  is the variable for the internal energy degrees of freedom, $\eta_{l_k} \in \mathbb{R}^{M}$ coincides with $\eta$ in the components corresponding to the internal degrees of freedom of species $k$ and is zero in the other components. For example, we can consider two species both composed of molecules consisting of two atoms, such that the molecules have rotational degrees of freedom in addition to the three translational degrees of freedom. In general, a molecule consisting of two atoms has three possible axes around which it can rotate. But since the energy needed to rotate the molecule around the axes parallel to the line connecting the two atoms is very high (see for example \cite{Kelly}), this does not occur, so we have two rotational degrees of freedom. In this example we have $M=l_1=l_2=2$. \\ Since we want to describe two different species, our kinetic model has two distribution functions $f_1(x,v,\eta_{l_1},t)> 0$ and $f_2(x,v,\eta_{l_2},t) > 0$. 
 Furthermore we relate the distribution functions to  macroscopic quantities by mean-values of $f_k$, $k=1,2$ as follows
\begin{align}
\int f_k(v, \eta_{l_k}) \begin{pmatrix}
1 \\ v \\ \eta_{l_k} \\ m_k |v-u_k|^2 \\ m_k |\eta_{l_k} - \bar{\eta}_k |^2 \\ m_k (v-u_k(x,t)) \otimes (v-u_k(x,t))
\end{pmatrix} 
dv d\eta_{l_k}=: \begin{pmatrix}
n_k \\ n_k u_k \\ n_k \bar{\eta}_k \\ d n_k T_k^{t} \\ l_k n_k T_k^{r} \\ \mathbb{P}_k
\end{pmatrix} , \quad k=1,2,
\label{moments}
\end{align} 
where $n_k$ is the number density, $u_k$ the mean velocity,  $T_k^{t}$ the mean temperature of the translation, $T_k^{r}$ the mean temperature of the internal energy degrees of freedom for example rotation or vibration and $\mathbb{P}_k$ the pressure tensor of species $k$, $k=1,2$. Note that in this paper we shall write $T_k^{t}$ and $T_k^{r}$ instead of $k_B T_k^{t}$ and $k_B T_k^{r}$, where $k_B$ is Boltzmann's constant. In  the following, we will require $\bar{\eta}_k=0$, which means that the energy in rotations clockwise is the same as in rotations counter clockwise. Similar for vibrations.

The distribution functions are determined by two equations to describe their time evolution. Furthermore we only consider binary interactions. 
So the particles of one species can interact with either themselves or with particles of the other species. In the model this is accounted for by introducing two interaction terms in both equations. These considerations allow us to write formally the system of equations for the evolution of the mixture. The following structure containing a sum of the collision operators is also given in \cite{Cercignani, Cercignani_1975}. \\
We are interested in a BGK approximation of the interaction terms. This leads us to define equilibrium distributions not only for each species itself but also for the two interspecies  distributions.  Then the model can be written as:

\begin{align} \begin{split} \label{BGK}
\partial_t f_1 + \nabla_x \cdot (v f_1)   &= \nu_{11} n_1 (M_1 - f_1) + \nu_{12} n_2 (M_{12}- f_1),
\\ 
\partial_t f_2 + \nabla_x \cdot (v f_2) &=\nu_{22} n_2 (M_2 - f_2) + \nu_{21} n_1 (M_{21}- f_2), 
\end{split}
\end{align}
with the Maxwell distributions
\begin{align} 
\begin{split}
M_k(x,v,\eta_{l_k},t) &= \frac{n_k}{\sqrt{2 \pi \frac{\Lambda_k}{m_k}}^d } \frac{1}{\sqrt{2 \pi \frac{\Theta_k}{m_k}}^{l_k}} \exp({- \frac{|v-u_k|^2}{2 \frac{\Lambda_k}{m_k}}}- \frac{|\eta_{l_k}|^2}{2 \frac{\Theta_k}{m_k}}), 
\\
M_{kj}(x,v,\eta_{l_k},t) &= \frac{n_{kj}}{\sqrt{2 \pi \frac{\Lambda_{kj}}{m_k}}^d } \frac{1}{\sqrt{2 \pi \frac{\Theta_{kj}}{m_k}}^{l_k}} \exp({- \frac{|v-u_{kj}|^2}{2 \frac{\Lambda_{kj}}{m_k}}}- \frac{|\eta_{l_k}|^2}{2 \frac{\Theta_{kj}}{m_k}}), 
\end{split}
\label{BGKmix}
\end{align}
for $ j,k =1,2, j \neq k$, 
where $\nu_{11} n_1$ and $\nu_{22} n_2$ are the collision frequencies of the particles of each species with itself, while $\nu_{12} n_2$ and $\nu_{21} n_1$ are related to interspecies collisions. \textcolor{black}{In this model, the  collision  frequencies  between  interspecies  operators can be  taken
different allowing molecular mass discrepancies. This point is relevant for example in the context of plasma physics, because the mass ratio between electrons and ions is very small.} 
To be flexible in choosing the relationship between the collision frequencies, we now assume the relationship
\begin{equation} 
\nu_{12}=\varepsilon \nu_{21}, \quad 0 < \frac{l_1}{l_1+l_2}\varepsilon \leq 1.
\label{coll}
\end{equation}
The restriction $\frac{l_1}{l_1+l_2} \varepsilon \leq 1$ is without loss of generality.
If $\frac{l_1}{l_1+l_2}\varepsilon >1$, exchange the notation $1$ and $2$ and choose $\frac{1}{\varepsilon}.$ In addition, we assume that all collision frequencies are positive.

Since rotational/vibrational and translational degrees of freedom relax at a different rate, $T_k^{t}$ and $T_k^{r}$ will first relax to partial temperatures $\Lambda_k$ and $\Theta_k$ respectively. Conservation of internal energy then requires that at each time
\begin{align}
\frac{d}{2} n_k \Lambda_k = \frac{d}{2} n_k T_k^{t} +\frac{l_k}{2} n_k T_k^{r} - \frac{l_k}{2} n_k \Theta_k, \quad k=1,2. \label{internal}
\end{align}  
Thus, $\Lambda_k$ can be written as a function of $\Theta_k.$
In equilibrium we expect the two temperatures $\Lambda_k$ and $\Theta_k$ to coincide, so we close the system by adding the equations
 \begin{align}
 \begin{split}
 \partial_t M_k + v \cdot \nabla_x M_k = \frac{\nu_{kk} n_k}{Z_r^k} \frac{d+l_k}{d} (\widetilde{M}_k - M_k)&+ \nu_{kk} n_k (M_k -f_k) \\&+ \nu_{kj} n_j (M_{kj} - f_k) ,
 \end{split}
 \label{kin_Temp}
 \end{align}
for $j,k=1,2, j \neq k$, where $Z_r^k$ are given parameters corresponding to the different rates of decays of translational and rotational/vibrational degrees of freedom. 
 Here $M_k$ is given by
\begin{align} 
M_k(x,v,\eta_{l_k},t) = \frac{n_k}{\sqrt{2 \pi \frac{\Lambda_k}{m_k}}^d } \frac{1}{\sqrt{2 \pi \frac{\Theta_k}{m_k}}^{l_k}} \exp({- \frac{|v-u_k|^2}{2 \frac{\Lambda_k}{m_k}}}- \frac{|\eta_{l_k}|^2}{2 \frac{\Theta_k}{m_k}}), \quad k=1,2,
\label{Maxwellian}
\end{align}
and $\widetilde{M}_k$ is given by 
\begin{align}
\widetilde{M}_k= \frac{n_k}{\sqrt{2 \pi \frac{T_k}{m_k}}^{d+l_k}} \exp \left(- \frac{m_k |v-u_k|^2}{2 T_k}- \frac{m_k|\eta_{l_k}|^2}{2 T_k} \right), \quad k=1,2.
\label{Max_equ}
\end{align}
where $T_k$ is the total equilibrium temperature and is given by 
\begin{align}
T_k:= \frac{d \Lambda_k + l_k \Theta_k}{d+l_k}= \frac{d T^{t}_k + l_k T^{r}_k}{d+l_k}.
\label{equ_temp}
\end{align}
The second equality follows from \eqref{internal}.
If we multiply \eqref{kin_Temp} by $|\eta_{l_k}|^2$, integrate with respect to $v$ and $\eta_{l_k}$ and use \eqref{equ_temp}, we obtain  
\begin{align}
\begin{split}
\partial_t(n_k \Theta_k) +   \nabla_x\cdot (n_k \Theta_k u_k) = \frac{\nu_{kk} n_k}{Z_r^k} n_k (\Lambda_k - \Theta_k)&+ \nu_{kk} n_k n_k (\Theta_k - T_k^{r}) \\ &+ \nu_{kj} n_j n_k(\Theta_{kj} - T_k^{r}) ,\quad k=1,2.
\end{split}
\label{relax}
\end{align}
 We obtained a macroscopic equation which describes the relaxation of the temperature $\Theta_k$ towards the temperature $\Lambda_k$ and the relaxation of $\Theta_k$ towards the rotational and vibrational temperature $T_k^{r}$ and of $T_k^{r}$ relaxing towards the mixture temperature $\Theta_{kj}$ in accordance with equation \eqref{BGK}. 
Note that  equation \eqref{relax} together with mass, momentum and total energy conservation, is equivalent to \eqref{kin_Temp}. 
In addition,  \eqref{BGK} and \eqref{kin_Temp} are consistent. If we multiply the equations for species $k$ of \eqref{BGK} and \eqref{kin_Temp}  by $v$ and integrate with respect to $v$ and $\eta_{l_k}$, we get in both cases for the right-hand side 
$$ \nu_{kj} n_j n_k  (u_{jk} - u_k),$$ and if we compute the total internal energy of both equations, we obtain in both cases $$ \frac{1}{2} \nu_{kj} n_k n_j [d \Lambda_{jk} + l_j \Theta_{jk} - ( d \Lambda_j + l_j \Theta_j)].$$
We will see this in \textcolor{black}{section \ref{sec3}} in theorem \ref{consenergy}. \\
 \\
We recall that we assume that the mean values of the momentum due to the  internal degrees of freedom $\bar{\eta}_1$, $\bar{\eta}_2$, $\bar{\eta}_{12}$ and $\bar{\eta}_{21}$ are zero.  The structure of the collision terms ensures that at equilibrium or when $\nu_{kj} \rightarrow \infty$ the distribution functions become Maxwell distributions.
With this choice of the Maxwell distributions $M_1$ and $M_2$ have the same densities, mean velocities and internal energies as $f_1$ respective $f_2$. This guarantees the conservation of mass, momentum and energy in interactions of one species with itself. 
The remaining parameters $n_{12}, n_{21}, u_{12}, u_{21}, \Lambda_{12}$ , $\Lambda_{21}$, $\Theta_{12}$ and $\Theta_{21}$ will be determined further down using conservation of the number of particles, total momentum and total energy, together with some symmetry considerations.
\textcolor{black}{We will determine $n_{12}$ and $n_{21}$ in equation \eqref{density} using conservation of the number of particles. The velocities $u_{12}$ and $u_{21}$ will be determined in equations \eqref{convexvel} and \eqref{veloc} by using conservation of total momentum. Last, the parameters  $\Lambda_{12}$ , $\Lambda_{21}$, $\Theta_{12}$ and $\Theta_{21}$ will be determined in theorem \ref{consenergy} and remark \ref{detpar}.  }
\\ \\
\textcolor{black}{Now, for the convenience of the reader, we want to write down our model again that one sees on the first view which equations we want to couple. Our BGK model for two species coupled with one relaxation equation and one algebraic equation for the temperatures can be written as
\begin{align*} \begin{split} 
\partial_t f_1 + \nabla_x \cdot (v f_1)   &= \nu_{11} n_1 (M_1 - f_1) + \nu_{12} n_2 (M_{12}- f_1),
\\ 
\partial_t f_2 + \nabla_x \cdot (v f_2) &=\nu_{22} n_2 (M_2 - f_2) + \nu_{21} n_1 (M_{21}- f_2), 
\end{split}
\end{align*}
\begin{align*}
\frac{d}{2} n_k \Lambda_k = \frac{d}{2} n_k T_k^{t} +\frac{l_k}{2} n_k T_k^{r} - \frac{l_k}{2} n_k \Theta_k,  
\end{align*}
\begin{align*}
\begin{split}
\partial_t(n_k \Theta_k) +   \nabla_x\cdot (n_k \Theta_k u_k) = \frac{\nu_{kk} n_k}{Z_r^k} n_k (\Lambda_k - \Theta_k)&+ \nu_{kk} n_k n_k (\Theta_k - T_k^{r}) \\ &+ \nu_{kj} n_j n_k(\Theta_{kj} - T_k^{r}) ,
\end{split}
\end{align*} for $k,j=1,2, k\neq j.$}

\section{Extension to an ES-BGK model}
\label{sec2}
\textcolor{black}{In this section again we first want to motivate how our model with several coupled equations will look like. For the convenience of the reader, we will summarize all this equations again at the end of the section such that one sees the whole model at a glance.}
It is well known that a drawback of the BGK approximation is its incapability of reproducing the correct Boltzmann hydrodynamic regime in the asymptotic continuum limit. Therefore, a modified version called ES-BGK model was suggested  by Holway in the case of one species \cite{Holway}. In this standard ES-BGK model, in the Maxwellian $M_k$, the scalar temperature $T_k^t$ related to the distribution function $f_k$ will be replaced by a linear combination of the temperature $T_k^t$ and the pressure tensor $\mathbb{P}_k$. In the polyatomic case described in this paper the translational temperature $T_k^t$ is different from the temperature $\Lambda_k$ of the Maxwellian $M_k$ given by \eqref{Maxwellian}. Now, we want to extend this temperature $\Lambda_k$ to a tensor $\Lambda_k^{ten}$ with $\text{trace}(\Lambda_k^{ten})=n_k \Lambda_k$ such that again we can consider a linear combination of the temperature $\Lambda_k$ and the tensor $\Lambda_k^{ten}$. In the BGK case described in the previous section we determined the time evolution of $\Theta_k$ by considering equation \eqref{kin_Temp} with the Maxwellian $M_k$ given by \eqref{Maxwellian} and the Maxwellian $\widetilde{M}_k$ given by \eqref{Mtilde} with the total equilibrium temperature $T_k$ given by \eqref{equ_temp} which leads to a time evolution of $\Theta_k$ given by \eqref{relax}. $\Lambda_k$ is then obtained by $\eqref{internal}$. Now, in the ES-BGK case
\textcolor{black}{we determine the time evolution of $f_k$ in the ES-BGK case by
\begin{align} \begin{split} \label{ESBGK}
\partial_t f_k + \nabla_x \cdot (v f_k)   &= \nu_{kk} n_k (G_k(f_k) - f_k) + \nu_{kj} n_j (M_{kj}(f_k,f_j)- f_k),
\end{split}
\end{align}
for $  k,j =1,2, k \neq j.$ 
To keep it as simple as possible we only replace the collision operators which represent the collisions of a species with itself by the ES-BGK collision operator for one species suggested in \cite{AndriesPerthame2001}. \textcolor{black}{The presence of a tensor in the attractors should allow to
overcome the well known problem of incorrect Prandtl number (analogously to paper [2]
for a single gas) but the proof is still lacking. We wanted to ensure that the is consistent with the special case of a single mono atomic gas.}  Other possible extensions are illustrated in the mono atomic case for gas mixtures in \cite{Pirner2}.For further references we denote the relaxation operators by $Q_{11}, Q_{12}, Q_{21}$ and $Q_{22}$.
Then we  define a function $G_k$ with a linear combination $\Lambda_k^{ES}$ given by
\begin{align*}
\Lambda_k^{ES} = (1- \mu_k) \Lambda_k \textbf{1}_d + \mu_k \frac{\Lambda_k^{ten}}{n_k}, \quad k=1,2.
\end{align*}
with $\mu_k \in \mathbb{R}$, $k=1,2$ being free parameters which we can choose in a way to fit physical parameters in the Navier-Stokes equations like the viscosity coefficient,
analogously as in the standard ES-BGK model given by
{\footnotesize
\begin{align} 
\begin{split}
G_k(f_k)(x,v,\eta_{l_k},t) = \frac{n_k}{\sqrt{\det(2 \pi \frac{\Lambda^{ES}_k}{m_k})} } \frac{1}{\sqrt{2 \pi \frac{\Theta_k}{m_k}}^{l_k}}  \exp\left({- \frac{1}{2} (v-u_k) \cdot \left(\frac{\Lambda_k^{ES}}{m_k}\right)^{-1} \cdot (v-u_k)}- \frac{1}{2} \frac{m_k |\eta_{l_k}|^2}{\Theta_k}\right) , 
\\
\end{split}
\label{ESBGKmix}
\end{align}}
for $k=1,2$. 
}

In order to determine the time evolution of $\Lambda_k^{ten}$ we consider the equation 
\begin{align}
\partial_t \widehat{G}_k + v \cdot \nabla_x \widehat{G}_k = \frac{\nu_{kk} n_k}{Z_r^k} \frac{d+l_k}{d} (\widetilde{G}_k - \widehat{G}_k) + \nu_{kk} n_k (G_k-f_k) + \nu_{kj} n_j (M_{kj} - f_k), \quad k=1,2,
\label{kin_TempES}
\end{align}
with the extended Maxwellian  $\widehat{G}_k$ given by 
\begin{align}
\widehat{G}_k= \frac{n_k}{\sqrt{\det( 2 \pi \frac{\Lambda_k^{ten}}{m_k})}} \frac{1}{\sqrt{2 \pi \frac{T_k^{r}}{m_k}}^{l_k}} \exp \left(- \frac{1}{2} (v-u_k) \cdot \left(\frac{\Lambda_k^{ten}}{m_k}\right)^{-1}\cdot (v- u_k)- \frac{m_k|\eta_{l_k}|^2}{2 \textcolor{black}{\Theta_k}} \right), 
\label{Max_equES}
\end{align}
for $k=1,2$, and the extended Maxwellian $\widetilde{G}_k$ given by
\begin{align}
\widetilde{G}_k = \frac{n_k}{\sqrt{\det(2\pi \frac{T_k^{ten}}{m_k}})} \frac{1}{\sqrt{2 \pi \frac{T_k}{m_k}}^{l_k}} \exp \left( - \frac{1}{2}(v-u_k) \cdot \left( \frac{T_k^{ten}}{m_k}\right)^{-1} \cdot (v-u_k) - \frac{1}{2}  \frac{m_k |\eta_{l_k}|^2 }{T_k}  \right).
\end{align}
The function $\widetilde{G}_k$ has the total equilibrium temperature $T_k$ and the pressure tensor of $f_k$ on the off-diagonals, namely
\begin{align}
\begin{split}
(T_k^{ten})_{ii} &= T_k \hspace{3.2cm}\text{for} \quad i=1,\dots d, \\
(T_k^{ten})_{ij} &= \frac{d}{d+l_k} \left(\frac{\mathbb{P}_k}{n_k} \right)_{ij} \hspace{1.4cm} \text{for} \quad i,j = 1, \dots d, i \neq j.
\end{split}
\label{Ten}
\end{align}
The factor $\frac{d}{d+l_k}$ in front of $\mathbb{P}_k$ in the definition of $T_k^{ten}$ has the following reason. The temperature \textcolor{black}{$T_k$} given by \eqref{equ_temp} is a convex combination of $T_k^{t}$ and $T_k^{r}$. Now, the off-diagonal elements of $T_k^{ten}$ have the same structure. It is a convex combination of the pressure tensor $\mathbb{P}_k$ and the tensor corresponding to the rotational and vibrational temperature. But since the rotational effects are diagonal, we have $(T_k^{ten})_{ij}=( \frac{d}{d+l_k} \mathbb{P} + \frac{l_k}{d+l_k} 0)_{\textcolor{black}{ij}}$ for $i \neq j$.

We only extended $\Lambda_k$ to a tensor and keep $\Theta_k$ as it is. This has the following reason. Since we assumed $\bar{\eta}_{lk}=0$, the microscopic velocities related to the internal degrees of freedom are symmetric and then we do not  distinguish different directions as we do in the translational degrees of freedom.

Equation \eqref{kin_TempES} leads to a time evolution of $\Lambda_k^{ten}$ given by 
\textcolor{black}{
\begin{align}
\begin{split}
\partial_t (n_k( \Lambda_k^{ten})_{ij}) &+   \nabla_x \cdot(n_k ( (\Lambda_k^{ten})_{ij})u_k) = \frac{\nu_{kk} n_k}{Z_r^k} \frac{d+l_k}{d} n_k( (T_k^{ten})_{ij} - (\Lambda_k^{ten})_{ij}) \\&+ \nu_{kk} n_k n_k ((\Lambda_k^{ES})_{ij} - (\mathbb{P}_k)_{ij}) + \nu_{kj} n_j n_k (\Theta_{kj} - T_k^{rot}) \delta_{ij},
\end{split}
\label{eqten}
\end{align}}
The evolution of $\Theta_k$ is then obtained from \eqref{internal}.   
\\ \\
\textcolor{black}{For the convenience of the reader we want to summarize our model again. In the case of the ES-BGK model we wanted to use ES-BGK equations for the mixture coupled with a relaxation equation and an algebraic equation for the temperatures.
\begin{align*} \begin{split} 
\partial_t f_k + \nabla_x \cdot (v f_k)   &= \nu_{kk} n_k (G_k(f_k) - f_k) + \nu_{kj} n_j (M_{kj}(f_k,f_j)- f_k),
\end{split}
\end{align*}
\begin{align*}
\frac{d}{2} n_k \Lambda_k = \frac{d}{2} n_k T_k^{t} +\frac{l_k}{2} n_k T_k^{r} - \frac{l_k}{2} n_k \Theta_k,  
\end{align*}
\begin{align*}
\begin{split}
\partial_t (n_k( \Lambda_k^{ten})_{ij}) &+   \nabla_x \cdot(n_k ( (\Lambda_k^{ten})_{ij})u_k) = \frac{\nu_{kk} n_k}{Z_r^k} \frac{d+l_k}{d} n_k( (T_k^{ten})_{ij} - (\Lambda_k^{ten})_{ij}) \\&+ \nu_{kk} n_k n_k ((\Lambda_k^{ES})_{ij} - (\mathbb{P}_k)_{ij}) + \nu_{kj} n_j n_k (\Theta_{kj} - T_k^{rot}) \delta_{ij},
\end{split}
\end{align*}for $  k,j =1,2, k \neq j.$
\\ \\ }


Since $G_k$ involves the term $(\Lambda_k^{ES})^{-1}$ and $\widetilde{G}_k$ involves the term $(T_k^{ten})^{-1}$ we have to check if $\Lambda_k^{ES}$ and $T_k^{ten}$ are invertible.
\begin{lemma}
Assume that $f_k$ and $\widehat{G}_k$ are positive solutions to \eqref{ESBGK} and \eqref{kin_TempES}.  Then $\Lambda_k^{ten}$ and $T_k^{ten}$  have strictly positive eigenvalues. Especially  $T_k^{ten}$ is invertible.
\end{lemma}
\begin{proof}
Let $y \in \mathbb{R}^d \setminus \lbrace 0 \rbrace$, then 
\begin{align*}
\langle y, \Lambda_k^{ten} y \rangle &= \sum_{i,j=1}^d  y_i ( \Lambda_k^{ten} )_{ij} y_j = \sum_{i,j=1}^d  y_i \int (v_i - u_{k,i})(v_j-u_{k,j}) \widehat{G}_k y_j dv \\ &= \int \left( \sum_{i,j=1}^d  y_i  (v_i - u_{k,i})\right)^2 \widehat{G}_k dv \geq 0.
\end{align*}
The inequality is true since we assumed that $\widehat{G}$ is a positive solution to \eqref{kin_Temp}. 

If we use equation \eqref{equ_temp} and \eqref{internal}
\begin{align*}
\langle y, T_k^{ten} y \rangle &= \sum_{i,j=1}^d  y_i ( T_k^{ten} )_{ij} y_j = \sum_{\stackrel{i,j=1}{i\neq j}}^d  y_i \int (v_i - u_{k,i})(v_j-u_{k,j}) f_k y_j dv + \sum_{i=1}^d y_i T_k y_j\\ &=\sum_{i,j=1}^d  y_i \int (v_i - u_{k,i})(v_j-u_{k,j}) f_k y_j dv - \sum_{i=1}^d y_i T_k^{t} y_i + \sum_{i=1}^d y_i \frac{d \Lambda_k + l_k \Theta_k}{d+l_k} y_i \\ &= \int \left( \sum_{i,j=1}^d  y_i  (v_i - u_{k,i})\right)^2 f_k dv  + \sum_{i=1}^{d} y_i T_k^{r} y_i \geq 0,
\end{align*}
where $T_k^{r}>0$ because $T_k^{r}$ is defined via a positive integral of $f_k$, see the definition in \eqref{moments}.
We even have strict inequality since $\lbrace y_i (v-u)_i\rbrace_{i=1}^d$ are linearly independent. 
\end{proof}

With the previous lemma, we can prove that $\Lambda_k^{ES}$ is positive. This is the next theorem. Positivity is also proven in \cite{AndriesPerthame2001} for the one species case, but for a different variant of an ES-BGK model.
\begin{theorem}
Assume that $f_k>0$ and \textcolor{black}{$ - \frac{1}{d-1} \leq \mu_k \leq 1$ if $d>1$.} Then $\Lambda_k^{ES}$ has strictly positive eigenvalues. Especially $\Lambda_k^{ES}$ is invertible.
\end{theorem}
\begin{proof}
Since $\Lambda_k^{ten}$ is symmetric there exist an invertible matrix $S_k$ such that $\widetilde{\Lambda_k^{ten}} = S_k \Lambda_k^{ten} S_k^{-1}$ with a diagonal matrix $\widetilde{\Lambda_k^{ten}}$. Then $\widetilde{\Lambda_k^{ES}}:= S_k \Lambda_k^{ES} S_k^{-1}$ is also diagonal since
$$\widetilde{\Lambda_k^{ES}} = S_k \Lambda_k^{ES} S_k^{-1} = (1-\mu_k) \Lambda_k \textbf{1} + \mu_k \widetilde{\Lambda_k^{ten}}.$$
Here we can see that the eigenvalues of $\widetilde{\Lambda_k^{ES}}$ are a linear combination of $\Lambda_k$ and the eigenvalues of $\widetilde{\Lambda_k^{ten}}$ which coincide with the eigenvalues of $\Lambda_k^{ten}$. We denote the eigenvalues of $\Lambda_k^{ten}$ by $\lambda_{k,1}, \lambda_{k,2}, \dots ,\lambda_{k,d}$. Then by definition of $\Lambda_k$ and $\Lambda_k^{ten}$ we have $$ d \Lambda_k = \textbf{Tr} (\Lambda_k^{ten} ) = \lambda_{k,1} + \lambda_{k,2} + \cdots + \lambda_{k,d}.$$
This means for the eigenvalues of $\Lambda_k^{ES}$ denoted by $\tau_{k,i}$:
$$ \tau_{k,i} = \frac{1- \mu_k}{d} \sum_{j=1}^d \lambda_{k,j} + \mu_k \lambda_{k,i} = \frac{1+\textcolor{black}{(d-1)}\mu_k}{d} \lambda_{k,i} + \frac{1- \mu_k}{d} \sum_{j=1, j \neq i}^d \lambda_{k,j}, \quad i=1,2,3.$$
Since $\lambda_{k,1}, \lambda_{k,2}, \dots, \lambda_{k,d}$ are strictly positive, the eigenvalues of $\Lambda_k^{ES}$ are strictly positive, when $1+ \textcolor{black}{(d-1)} \mu_k$ and $1- \mu_k$ are positive. Since we restricted $\mu_k$ to $-\frac{1}{\textcolor{black}{d-1}} \leq \mu_k \leq 1$ if $d>1$, $\Lambda_k^{ES}$ is strictly positive.
\end{proof}
\subsection{Conservation properties}
\label{sec3}
Conservation of the number of particles and total momentum of the model for mixtures described in \textcolor{black}{section \ref{sec1}} are shown in the same way as in the case of mono atomic molecules. In the extension described in \textcolor{black}{section \ref{sec2}} these
conservation properties are still satisfied since $G_1$ and $G_2$  have the same density, mean velocity and internal energy as $f_1$ respective $f_2$. Conservation of the number of particles and of total momentum are guaranteed by the following choice of the mixture parameters:\\ \\
If we assume that \begin{align} n_{12}=n_1 \quad \text{and} \quad n_{21}=n_2,  
\label{density} 
\end{align}
we have conservation of the number of particles, see Theorem 2.1 in \cite{Pirner}.
If we further assume that $u_{12}$ is a linear combination of $u_1$ and $u_2$
 \begin{align}
u_{12}= \delta u_1 + (1- \delta) u_2, \quad \delta \in \mathbb{R},
\label{convexvel}
\end{align} then we have conservation of total momentum
provided that
\begin{align}
u_{21}=u_2 - \frac{m_1}{m_2} \varepsilon (1- \delta ) (u_2 - u_1),
\label{veloc}
\end{align}
see Theorem 2.2 in \cite{Pirner}.

In the case of total energy we have a difference for the polyatomic case compared to the monoatomic one. So we explicitly consider this in the following theorem.
\begin{theorem}[Conservation of total energy]
Assume \eqref{coll}, conditions \eqref{density}, \eqref{convexvel} and \eqref{veloc} and assume that $\Lambda_{12}$ and $\Theta_{12}$ are of the following form
\begin{align}
\begin{split}
\Lambda_{12} &=  \alpha \Lambda_1 + ( 1 - \alpha) \Lambda_2 + \gamma |u_1 - u_2 | ^2,  \quad 0 \leq \alpha \leq 1, \gamma \geq 0, \\ \Theta_{12}&=  \frac{l_1 \Theta_1 + l_2 \Theta_2}{l_1+l_2}. 
\label{contemp}
\end{split}
\end{align}
Then we have conservation of total energy
\begin{align*}
\int \frac{m_1}{2} (|v|^2 + |\eta_{l_1}|^2 )(Q_{11}(f_1,f_1)+Q_{12}(f_1,f_2)) dv d\eta_{l_1} \\+
\int \frac{m_2}{2} (|v|^2 + |\eta_{l_2}|^2 )(Q_{22}(f_2,f_2)+Q_{21}(f_2,f_1)) dv d\eta_{l_2}= 0,
\end{align*}
\label{consenergy}
\end{theorem}
provided that
\begin{align}
\begin{split}
\Lambda_{21} + \frac{l_2}{d} \Theta_{21}=\left[ \frac{1}{d} \varepsilon m_1 (1- \delta) \left( \frac{m_1}{m_2} \varepsilon ( \delta - 1) + \delta +1 \right) - \varepsilon \gamma \right] |u_1 - u_2|^2 \\+ \varepsilon ( 1 - \alpha ) \Lambda_1 + ( 1- \varepsilon ( 1 - \alpha)) \Lambda_2 + \frac{1}{d} \varepsilon \frac{l_1 l_2}{l_1+l_2} \Theta_1 + \frac{1}{d} (l_2- \varepsilon \frac{l_1 l_2}{l_1+l_2} ) \Theta_2.
\label{temp}
\end{split}
\end{align}
\begin{proof}
Using the definition of the energy exchange of species $1$ and equation \eqref{internal}, we obtain
\begin{align*}
F_{E_{1,2}}:&= \int \frac{m_1}{2} (|v|^2 + |\eta_{l_1}|^2) [Q_{11}(f_1,f_2)+ Q_{12}(f_1,f_2)] dv d\eta_{l_1} 
\\&= \varepsilon \nu_{21} \frac{1}{2} n_2 n_1 m_1( |u_{12}|^2 - |u_1|^2)+ \frac{d}{2} \varepsilon \nu_{21} n_1 n_2  (\Lambda_{12} - T_1^{t})  +\frac{l_1}{2} \varepsilon \nu_{21} n_1 n_2 ( \Theta_{12} - T_1^{r} ) \\&= \varepsilon \nu_{21} \frac{1}{2} n_2 n_1 m_1( |u_{12}|^2 - |u_1|^2)+ \frac{d}{2} \varepsilon \nu_{21} n_1 n_2  (\Lambda_{12} - \Lambda_1)  +\frac{l_1}{2} \varepsilon \nu_{21} n_1 n_2 ( \Theta_{12} - \Theta_1 ).
\end{align*}
 Next, we will insert the definitions of $u_{12}$, $\Lambda_{12}$ and $\Theta_{12}$ given by \eqref{convexvel} and \eqref{contemp}. Analogously the energy exchange of species $2$ towards $1$ is 
$$
F_{E_{2,1}}=  \nu_{21} \frac{1}{2} n_2 n_1 m_2( |u_{21}|^2 - |u_2|^2)+ \frac{d}{2} \nu_{21} n_1 n_2  (\Lambda_{21} - \Lambda_2)  + \frac{l_2}{2} \nu_{21} n_1 n_2 ( \Theta_{21} - \Theta_2 ).
$$
Substitute $u_{21}$ with \eqref{veloc} and $\Lambda_{21} + \frac{l_2}{d}\Theta_{21}$ from \eqref{temp}. This permits to rewrite the energy exchange as 
\begin{align}
\begin{split}
F_{E_{1,2}}= \varepsilon \nu_{21} \frac{1}{2} n_2 n_1 m_1 (1-\delta) \left[ (u_1+u_2) - \delta(u_2-u_1)\right] (u_1-u_2)\\ + \frac{1}{2} \varepsilon \nu_{21} n_1 n_2  \left[(1-\alpha) d (\Lambda_2 - \Lambda_1)  + \frac{l_1 l_2}{l_1+l_2} (\Theta_2 -  \Theta_1) + \gamma d |u_1-u_2|^2\right],
\end{split}
 \label{flux_en_12}
\end{align}
\begin{align}
\begin{split}
F_{E_{2,1}} =\textcolor{black}{\frac{1}{2} \nu_{21} m_2 n_1 n_2  \big[ \left( (1-\frac{m_1}{m_2} \varepsilon (1- \delta) )^2 -1 \right) |u_2|^2 
+ \left( \frac{m_1}{m_2} \varepsilon (\delta - 1) \right)^2 |u_1|^2}  \\
+ 2 ( 1- \frac{m_1}{m_2} \varepsilon (1-\delta)) \frac{m_1}{m_2} \varepsilon ( 1- \delta) u_1 \cdot u_2 
\big] 
+ \frac{1}{2} \nu_{21} n_1 n_2 \big[ \varepsilon ( 1- \alpha) d( \Lambda_1 - \Lambda_2)\\+ \varepsilon \frac{l_1 l_2}{l_1+l_2}(\Theta_1-\Theta_2)  + \left(  \varepsilon m_1 (1- \delta) \left( \frac{m_1}{m_2} \varepsilon ( \delta - 1) + \delta +1 \right) - \varepsilon \gamma d \right) |u_1 - u_2|^2 \big].
\end{split}
\label{flux_en_21}
\end{align}
Adding these two terms, we see that the total energy is conserved.
\end{proof}
\begin{remark}
The energy flux between the two species is zero if and only if $u_1=u_2,$ $\Lambda_1=\Lambda_2,$ $\Theta_1=\Theta_2$ provided that $\alpha, \delta <1$ and $\gamma >0$.
\end{remark}
\begin{remark}
From conservation of total energy we get only one condition on $\Lambda_{21} + \textcolor{black}{\frac{l_2}{d}} \Theta_{21}$ given by \eqref{temp}, but not an explicit formula for $\Lambda_{21}$ and $\Theta_{21}$. In order to keep the model symmetric we again separate the temperatures corresponding to the translational part and the one corresponding to the rotational and vibrational part and choose
\begin{align}
\begin{split}
\Lambda_{21} &= \varepsilon (1- \alpha) \Lambda_1 + (1- \varepsilon(1-\alpha)) \Lambda_2 \\ &+ \left[ \frac{1}{d} \varepsilon m_1 (1- \delta) \left( \frac{m_1}{m_2} \varepsilon ( \delta - 1) + \delta +1 \right) - \varepsilon \gamma \right] |u_1 - u_2|^2, \label{Lambda21}\end{split} \\
\Theta_{21} &= \left(1- \varepsilon \frac{l_1}{l_1+l_2} \right) \Theta_2 + \varepsilon \frac{l_1}{l_1+l_2} \Theta_1.
\label{Theta21}
\end{align}
\label{detpar}
\end{remark}
\begin{remark}
If $l_1=l_2$, we have $\Theta_{12}= \frac{1}{2} (\Theta_1 + \Theta_2)$. We then find $\Theta_{21}=\Theta_{12}$ if the two species have the same interspecies collision frequency ($\varepsilon=1$).
\end{remark}

\textcolor{black}{\begin{remark}The fact that we only consider the two species case is just for simplicity. We can also extend the model to more than two species, because we assume that we only have binary interactions. So if we consider collision terms given by $$\nu_{ii} n_i (G_i -f_i) + \sum_{j \neq i}^N \nu_{ij} n_j (G_{ij} - f_i), \quad i=1,...,N,$$ we expect that we have conservation of total momentum and total energy in every interaction of species $i$ with species $j$. This means we require
$$ \int \begin{pmatrix}
v \\ v^2
\end{pmatrix} \nu_{ij} n_j (G_{ij} - f_i) dv + \int \begin{pmatrix}
v \\ v^2
\end{pmatrix} \nu_{ji} n_i (G_{ji} - f_j) dv =0,$$ for every $i,j=1,...N, ~ i \neq j$ and so it reduces to the two species case. 
\end{remark}}

\subsection{Positivity of the temperatures}
\label{sec4}
\begin{theorem}
Assume that $f_1(x,v, \eta_{l_1},t), f_2(x,v,\eta_{l_2},t) > 0$. Then  all temperatures $\Lambda_1$, $\Lambda_2$, $\Theta_1$, $\Theta_2$,and $\Lambda_{12}$, $\Theta_{12}$ given by \eqref{contemp},  and $\Lambda_{21}$, $\Theta_{21}$ determined by \eqref{Lambda21}, \eqref{Theta21} are positive provided that 
 \begin{align}
0 \leq \gamma  \leq \frac{m_1}{d} (1-\delta) \left[(1 + \frac{m_1}{m_2} \varepsilon ) \delta + 1 - \frac{m_1}{m_2} \varepsilon \right].
 \label{gamma}
 \end{align}
\end{theorem}
\begin{proof}
The temperatures $\Lambda_1, \Lambda_2, \Theta_1, \Theta_2, \Lambda_{12}, \Theta_{12}$ and $\Theta_{21}$ are positive by definition because they are integrals or convex combinations of positive functions. So the only thing to check is when the temperature $\Lambda_{21}$ in \eqref{Lambda21} is positive. This is done in \cite{Pirner} for $d=3$, so we skip the proof here. The resulting condition is given by \eqref{gamma}.
\end{proof}
\begin{remark}
Since $\gamma \geq 0$ is a non-negative number, so the right-hand side of the inequality in \eqref{gamma} must be non-negative. This condition is equivalent to 
\begin{align}
 \frac{ \frac{m_1}{m_2}\varepsilon - 1}{1+\frac{m_1}{m_2}\varepsilon} \leq  \delta \leq 1.
\label{gammapos}
\end{align}
\end{remark}
Note that we have to assume that the distribution function $f_1$ and $f_2$ are positive. In \cite{Pirner3}, positivity of the distribution function for the model described in \cite{Pirner} for mono atomic molecules is proven. This method can be extended to the model described in this paper for polyatomic molecules.
\subsection{The structure of equilibrium}
\label{sec5}
\begin{theorem}[Equilibrium]
Assume $f_1, f_2 >0$ with $f_1$ and $f_2$ independent of $x$ and $t$.
Assume the conditions \eqref{density}, \eqref{convexvel}, \eqref{veloc}, \eqref{contemp} and \eqref{temp}, $\delta \neq 1, \alpha \neq 1, l_1,l_2\neq 0$, so that all temperatures are positive. 

Then $f_1$ and $f_2$ are Maxwell distributions with equal mean velocities $u_1=u_2=u_{12}=u_{21}$ and temperatures $T:=T_1^{r}=T_2^{r}=T_1^{t}=T_2^{t}=\Lambda_1=\Lambda_2=\Theta_1=\Theta_2=\Theta_{12}=\Theta_{21}= \Lambda_{12}=\Lambda_{21}$. \textcolor{black}{This means $f_k$ is given by
$$ M_k(x,v,\eta_{l_k},t) = \frac{n_k}{\sqrt{2 \pi \frac{T}{m_k}}^d } \frac{1}{\sqrt{2 \pi \frac{T}{m_k}}^{l_k}} \exp({- \frac{|v-u|^2}{2 \frac{T}{m_k}}}- \frac{|\eta_{l_k}|^2}{2 \frac{T}{m_k}}), \quad k=1,2.$$}
\label{equilibrium}
\end{theorem}
\begin{proof}

Equilibrium means that $f_1,f_2, \Lambda_1, \Lambda_2, \Theta_1, \Theta_2$ are independent of $x$ and $t$. 
 Thus in equilibrium the right-hand side of the equations \eqref{ESBGK}  and \eqref{kin_TempES} have to be zero. In particular,
\begin{align}
(\nu_{11} n_1 + \nu_{12} n_2) f_1 &= \nu_{11} n_1 G_1 + \nu_{12} n_2 M_{12},\label{1} \\
(\nu_{22} n_2 + \nu_{21} n_1)f_2 &= \nu_{22} n_2 G_2 + \nu_{21} n_1 M_{21}. \label{2} 
\end{align}
Since the right-hand side of \eqref{ESBGK} and the right-hand side of \eqref{kin_TempES} have to be zero, the difference of the right-hand side of \eqref{ESBGK} and the right-hand side of \eqref{kin_TempES} has to be equal to zero. If we compute the translational temperature of this difference, we obtain
\begin{align}
\Lambda_1^{ten} &= T_1^{ten}, \label{1+}\\
\Lambda_2^{ten} &= T_2^{ten}. \label{2+}
\end{align}
Especially, from the diagonal part of  \eqref{1+} and \eqref{2+} we can deduce
\begin{align}
\Lambda_1&=\Theta_1, \label{3}\\
\Lambda_2&= \Theta_2. \label{4}
\end{align}
When we consider the moment of the velocity of \eqref{1}, we get
$$(\nu_{11} n_1 + \nu_{12} n_2) u_1 = \nu_{11} n_1 u_1 + \nu_{12} n_2 u_{12}.$$
Substituting $u_{12}= \delta u_1 + (1- \delta) u_2$, we have
\begin{align}
u_1=u_2, \label{5}
\end{align}
for $\delta \neq 1$.
\\ Using \eqref{3}, \eqref{4} and \eqref{5}, the temperatures of the mixture Maxwellians \eqref{contemp} and \eqref{Lambda21}, \eqref{Theta21}  simplify to 
\begin{align}
\Lambda_{12}&= \alpha \Lambda_1 + (1- \alpha) \Lambda_2, \quad \Theta_{12}=  \frac{l_1}{l_1+l_2} \Lambda_1 + \frac{l_2}{l_1+l_2} \Lambda_2, \\
\Lambda_{21} &= \varepsilon (1- \alpha) \Lambda_1 + (1- \varepsilon(1-\alpha)) \Lambda_2, \quad
\Theta_{21} =  \varepsilon \frac{l_1}{l_1+l_2} \Lambda_1 + ( 1-  \varepsilon \frac{l_1}{l_1+l_2}) \Lambda_2.
\label{10}
\end{align}
When we consider the moments of the translational and the rotational and vibrational temperatures of \eqref{1} and \eqref{2}, we get
\begin{align}
(\nu_{11} n_1 + \nu_{12} n_2) T_1^{t} &= (\nu_{11} n_1 + \nu_{12} n_2 \alpha) \Lambda_1 + \nu_{12} n_2 (1-\alpha) \Lambda_2, \label{6} \\
(\nu_{11} n_1 + \nu_{12} n_2) T_1^{r} &= (\nu_{11} n_1 + \nu_{12} n_2 \frac{l_1}{l_1+l_2}) \Lambda_1 + \nu_{12} n_2 \frac{l_2}{l_1+l_2} \Lambda_2, \label{7} \\
(\nu_{22} n_2 + \nu_{21} n_1)T_2^{t} &= \nu_{22} n_2 \Lambda_2 + \nu_{21} n_1 \Lambda_{21}, \label{8} \\
(\nu_{22} n_2 + \nu_{21} n_1)T_2^{r} &= \nu_{22} n_2 \Lambda_2 + \nu_{21} n_1 \Theta_{21}, \label{9}
\end{align}
where we used the definitions of the mixture velocities and temperatures \eqref{density}, \eqref{convexvel}, \eqref{veloc}, \eqref{contemp} and equations \eqref{3}, \eqref{4} and \eqref{5}. 
Analogue, equations \eqref{internal} simplify to
\begin{align}
\frac{d+l_1}{2} \Lambda_1 = \frac{d}{2} T_1^{t} + \frac{l_1}{2} T_1^{r}, \label{11} \\
\frac{d+l_2}{2} \Lambda_2 = \frac{d}{2} T_2^{t} + \frac{l_2}{2} T_2^{r}. \label{12}
\end{align}
Inserting \eqref{6} and \eqref{7} in \eqref{11}, we obtain
\begin{align*}
\frac{d}{2} ( \frac{\nu_{11} n_1 + \nu_{12} n_2 \alpha}{\nu_{11} n_1 + \nu_{12} n_2} \Lambda_1 + \frac{\nu_{12} n_2 (1 - \alpha)}{\nu_{11} n_1 + \nu_{12} n_2} \Lambda_2 ) +\frac{l_1}{2} (\frac{\nu_{11} n_1 + \nu_{12} n_2 \frac{l_1}{l_1+l_2}}{\nu_{11} n_1 + \nu_{12} n_2} \Lambda_1 + \frac{\nu_{12} n_2 \frac{l_2}{l_1+l_2}}{\nu_{11} n_1 + \nu_{12} n_2} \Lambda_2 )\\= \frac{d+l_1}{2} \Lambda_1,
\end{align*}
which, provided $d \alpha + l_1 \frac{l_1}{l_1+l_2} \neq d+l_1$, is equivalent to 
\begin{align}
\Lambda_1 = \Lambda_2. \label{13}
\end{align}
This condition is equivalent to $d(1-\alpha) + \frac{l_1 l_2}{l_1+l_2} \neq 0$ which is satisfied since $\alpha \neq 1, l_1,l_2 \neq 0$.
With \eqref{13} we can deduce from \eqref{6} and \eqref{7} that
\begin{align}
T_1^{t} = \Lambda_1 \quad \text{and} \quad T_1^{r} = \Lambda_1.
\label{13+}
\end{align}
Condition \eqref{10} together with \eqref{13} leads to 
\begin{align}
\Lambda_{21} = \Theta_{21} = \Lambda_1. \label{14}
\end{align}
Inserting \eqref{13} and \eqref{14} in \eqref{8} and \eqref{9} leads to
\begin{align}
T_2^{t} = T_2^{r} = \Lambda_1.
\end{align}
If we compute the pressure tensor of \eqref{1} using that all temperatures are equal to $\Lambda_1$ we obtain 
$$(\nu_{11} n_1 + \nu_{12} n_2) \frac{\mathbb{P}_1}{n_1} = \nu_{11} n_1 (1- \mu_1) \Lambda_1 \textbf{1} + \nu_{11} n_1 \mu_1 \Lambda_1^{ten} + \nu_{12} n_2 \Lambda_1 \textbf{1}.$$
Using \eqref{Ten}, \eqref{1+} and \eqref{13+},
 we have that \textcolor{black}{$$\frac{d}{d+l_k}\frac{\mathbb{P}_1}{n_1 }+ \frac{l_k}{d+l_k} \Lambda_1^{rot} \textbf{1}_d=  \Lambda_1^{ten} =  T_1^{ten}$$}  and therefore 
$$ ( \nu_{11} n_1 (1- \mu_1 \textcolor{black}{\frac{d}{d+l_k}}) + \nu_{12} n_2)  \frac{\mathbb{P}_1}{n_1}  = ( \nu_{11} n_1 (1- \mu_1 \frac{d}{d+l_k}) + \nu_{12} n_2) \Lambda_1 \textbf{1} ,$$
 for $j \neq i$, which shows that the pressure tensor of $f_1$ is diagonal since $\mu_1 \leq 1.$ Similar for $\frac{\mathbb{P}_2}{n_2}$ using \eqref{2}, \eqref{2+} and \eqref{13+}.

So all in all, in equilibrium we get that $f_1$ and $f_2$ are Maxwell distributions with equal mean velocities $u_1=u_2=u_{12}=u_{21}$ and temperatures $T_1^{r}=T_2^{r}=T_1^{t}=T_2^{t}=\Lambda_1=\Lambda_2=\Theta_1=\Theta_2=\Theta_{12}=\Theta_{21}= \Lambda_{12}=\Lambda_{21}$.
\end{proof}
\begin{definition}
If $f_1$ and $f_2$ are Maxwell distributions with equal mean velocities $u=u_1=u_2$ and temperatures $T=T_1^{r}=T_2^{r}=T_1^{t}=T_2^{t}=\Lambda_1=\Lambda_2=\Theta_1=\Theta_2$, then we say that $f_1$ and $f_2$ are in local equilibrium.
\label{localequ}
\end{definition}
Note that for $\alpha=1$ or $\delta =1$, we have no exchange of momentum and energy of the tow species, so we do not expect a relaxation towards a common equilibrium. So in the following, we always assume $\alpha, \delta \neq 1$.
\subsection{H-Theorem}
\label{sec6}
In this section we will prove that our model admits an entropy with an entropy inequality. For this, we have to prove an inequality on the term $\int \ln f_k (G_k - f_k) dv d\eta_{l_k}$ coupled with the right-hand side of equation \eqref{kin_TempES} and an inequality on $\nu_{12} n_2 \int (M_{12} - f_1) \ln f_1 dv d\eta_{l_1} + \nu_{21} n_1 \int (M_{21} - f_2) \ln f_2 dv d\eta_{l_2}$ coupled with the right-hand side of equation \eqref{kin_TempES}. We prove the first one in subsection \ref{sec6.1} and the second one in subsection \ref{sec6.2}.
\subsubsection{H-Theorem for the one species relaxation terms}
\label{sec6.1}
\begin{remark}
From the definition of the moments of $f_k, k=1,2$ in \eqref{moments} and the definitions of the extended Maxwellians $G_k, k=1,2$ in \eqref{ESBGKmix}, we see that the \textcolor{black}{pressure tensors} and the temperatures, do not coincide. Now, we consider extended Maxwellians $\bar{G}_k, k=1,2$ which have the same moments as $f_k, k=1,2$.
Then from the case of one species ES-BGK model we know that
$$ \int \bar{G}_k \ln \bar{G}_k dv d\eta_{l_k} \leq \int f_k \ln f_k dv d\eta_{l_k},$$ for $k=1,2$, see equations $(20)$ and $(21)$ in \cite{AndriesPerthame2001} in the mono atomic case. The polyatomic case is analogously to the mono atomic case.
\label{one}
\end{remark}
\begin{lemma}
Assume that $f_1,f_2 >0$. As in remark \ref{one} let $\bar{G}_k$ be the extended Maxwellians with the same moments as $f_k, k=1,2$ and $\widetilde{G}_k$ the Maxwellians defined by \eqref{Max_equES}. Then we have \begin{align*}
\int \widetilde{G}_k \ln \widetilde{G}_k dv d\eta_{l_k} &\leq \int \bar{G}_k \ln \bar{G}_k dv d\eta_{l_k}, \quad k=1,2, \\\int \widehat{G}_k \ln \widehat{G}_k dv d\eta_{l_k} &\geq \int G_k \ln G_k dv \eta_{l_k}, \quad k=1,2 ,\\ \int G_k \ln G_k dv d\eta_{l_k} &\geq \int M_k \ln M_k dv \eta_{l_k}, \quad k=1,2. 
\end{align*} 
\label{Mtilde}
\end{lemma}
\begin{proof}
The proof of the second inequality is analogously to the proof in the mono atomic case of equation $(21)$ in \cite{AndriesPerthame2001}. So we only prove the first and the third one.
Using that {\small \\$\ln M_k = \ln \left(\frac{n_k}{\sqrt{2 \pi \frac{\Lambda_k}{m_k}}^d} \frac{1}{\sqrt{2 \pi \frac{\Theta_k}{m_k}}^{l_k}} \right) - \frac{|v-u_k|^2}{2 \frac{\Lambda_k}{m_k}} - \frac{|\eta_{l_k}|^2}{2 \frac{\Theta_k}{m_k}}$,
\\ $\ln \bar{G}_k = \ln \left(\frac{n_k}{\sqrt{\det \left(2 \pi \frac{\mathbb{P}_k}{m_k} \right)}} \frac{1}{\sqrt{2 \pi \frac{T^{r}_k}{m_k}}^{l_k}} \right) - \frac{1}{2} \left(v-u_k \right) \cdot \left( \frac{\mathbb{P}_k}{m_k} \right)^{-1} \cdot \left(v-u_k \right) - \frac{|\eta_{l_k}|^2}{2 \frac{T^{r}_k}{m_k}}$,
\\ $\ln \widetilde{G}_k = \ln (\frac{n_k}{\sqrt{\det(2 \pi \frac{T_k^{ten}}{m_k})}}\textcolor{black}{\frac{1}{\sqrt{2 \pi \frac{T_k}{m_k}}^{l_k}} }) - \frac{1}{2} (v-u_k) \cdot \left( \frac{T_k^{ten}}{m_k}\right)^{-1} \cdot (v-u_k) - \frac{|\eta_{l_k}|^2}{2 \frac{T_k}{m_k}}, $ 
~~~and \\ $\ln G_k = \ln \left(\frac{n_k}{\sqrt{\det \left(2 \pi \frac{\Lambda_k^{ES}}{m_k} \right)}} \frac{1}{\sqrt{2 \pi \frac{\Theta_k}{m_k}}^{l_k}} \right) - \frac{1}{2} m_k \left(v-u_k \right) \cdot \left( \Lambda_k^{ES} \right)^{-1} \cdot \left(v-u_k \right) - \frac{|\eta_{l_k}|^2}{2 \frac{\Theta_k}{m_k}}$,} \\ we compute the integrals and obtain that the required inequalities are equivalent to
\begin{align*}
\ln\left(\frac{n_k}{\sqrt{\det(2 \pi \frac{T_k^{ten}}{m_k})}} \frac{1}{\sqrt{2 \pi \frac{T_k}{m_k}}^{l_k}} \right) &\leq \ln\left(\frac{n_k}{\sqrt{\det(2 \pi \frac{\mathbb{P}_k}{m_k})}} \frac{1}{\sqrt{2 \pi \frac{T^{r}_k}{m_k}}^{l_k}}\right), \\
\ln\left(\frac{n_k}{\sqrt{\det(2 \pi \frac{\Lambda_k^{ES}}{m_k})}}\right) &\geq \ln \left(\frac{n_k}{\sqrt{ 2 \pi \frac{\Lambda_k}{m_k}}^d}\right).
\end{align*} This is equivalent to the conditions
\begin{align}
\begin{split}
\ln \det ( T_k^{ten}) + l_k \ln T_k &\geq \ln \det \mathbb{P}_k + l_k \ln T_k^{r}, \\
(\Lambda_k)^d &\geq \det (\Lambda_k^{ES}).
\end{split}
\label{inequtrace}
\end{align} 
We first look at the first inequality. If we insert the expression for $T_k$ given by \eqref{equ_temp} and use the concavity of $\ln$, we obtain
\begin{align}
\begin{split}
\ln \det ( T_k^{ten}) + l_k \frac{l_k}{d+l_k} \ln T_k^{r} + l_k \frac{d}{d+l_k} \ln T_k^{t} \geq \ln \det \mathbb{P}_k + l_k \ln T_k^{r}.
\end{split}
\label{bm1}
\end{align} 
Now we use the Brunn-Minkowsky inequality (inequality $(27)$ in \cite{AndriesPerthame2001}) given by
$$ \det(a A + (1-a) B) \geq (\det A)^a (\det B)^{1-a},$$
for $0 \leq a \leq 1$ and $A,B$ positive symmetric matrices. Since we can write $T_k^{ten}$ as
$$T_k^{ten}= \frac{d}{d+l_k} \mathbb{P}_k + \frac{l_k}{d+l_k} T_k^{r} \textbf{1}_d,$$ we can apply the Brunn-Minkowsky inequality on \eqref{bm1} and obtain
\begin{align*}
\frac{d}{d+l_k} \ln \det \mathbb{P}_k + d \frac{l_k}{d+l_k} \ln T_k^{r} + l_k \frac{l_k}{d+l_k} \ln T_k^{r} + l_k \frac{d}{d+l_k} \ln T_k^{t} \\\geq \ln \det \mathbb{P}_k + l_k \ln T_k^{r} 
\end{align*} 
So it remains to show that
$$(T_k^{t})^{d} \geq \det \mathbb{P}_k.$$
This inequality has the same structure as the second inequality in \eqref{inequtrace}. So we  prove only the second inequality in \eqref{inequtrace}. 
We observe that trace($\Lambda_k^{ES})=d \Lambda_k,$ so we have to show
$$ \left( \frac{\text{trace}(\Lambda_k^{ES})}{d}\right)^d \geq \det (\Lambda_k^{ES}).$$
Let $\lambda_1, \dots , \lambda_d$ the eigenvalues of the symmetric positive matrix $\Lambda_k^{ES}$, then this inequality is equivalent to 
$$ \left(\frac{\lambda_1+ \cdots + \lambda_d}{d}\right)^d \geq \lambda_1 \cdots \lambda_d.$$
This is true since it is the inequality of arithmetic and geometric means.
\end{proof}

\begin{lemma}[Contribution to the H-theorem from the one species relaxation terms]
Assume $f_1, f_2 >0$. 
Then
$$ \int \ln f_k (G_k - f_k) dv d\eta_{l_k} + 
\int \ln \widehat{G}_k ( \widetilde{G}_k - \widehat{G}_k) dv d\eta_{l_k} \leq 0, \quad k=1,2, $$ with equality if and only if $M_k=f_k$ and  $\Lambda_k=\Theta_k=T_k^{r} = T_k^{t}$.
\label{one_species}
\end{lemma}
\begin{proof}
Since the function $H(x)= x \ln x -x$ is strictly convex for $x>0$, we have $H'(f) (g-f) \leq H(g) - H(f)$ with equality if and only if $g=f$. So \begin{align}
(g-f) \ln f  \leq g \ln g - f \ln f +f -g. 
\label{convex1}
\end{align}
Apply \eqref{convex1} on both terms of $$S_k(f_k):= \int \ln f_k (G_k- f_k ) dv d\eta_{l_k} + 
 \int \ln \widehat{G}_k ( \widetilde{G}_k - \widehat{G}_k) dv d\eta_{l_k}. $$
 Then we obtain
\begin{align*}
S_k(f_k) \leq \int G_k \ln G_k dv d\eta_{l_k} - \int f_k \ln f_k dv d\eta_{l_k} - \int G_k dv d\eta_{l_k} + \int f_k dv d\eta_{l_k}\\ + 
 [ \int \widetilde{G}_k \ln \widetilde{G}_k dv d\eta_{l_k} - \int \widehat{G}_k \ln \widehat{G}_k dv d\eta_{l_k} - \int \widetilde{G}_k dv d\eta_{l_k} + \int \widehat{G}_k dv d\eta_{l_k}], 
\end{align*}
with equality if and only if $f_k=G_k$ and $G_k=\widetilde{G}_k$ from which we can deduce $f_k=M_k$ by computing macroscopic quantities of $f_k=G_k$ and $G_k=\widetilde{G}_k$.
Since $f_k$, $G_k$, $\widehat{G}_k$ and $\widetilde{G}_k$ have the same density, we obtain
\begin{align}
S(f_k) \leq \int G_k \ln G_k dv d\eta_{l_k} - \int f_k \ln f_k dv d\eta_{l_k}+ 
[ \int \widetilde{G}_k \ln \widetilde{G}_k dv d\eta_{l_k} - \int \widehat{G}_k \ln \widehat{G}_k dv d\eta_{l_k}].
\label{S}
\end{align}
According to the second part of lemma \ref{Mtilde}, we obtain
\begin{align}
S(f_k) &\leq 
 \int \widetilde{G}_k \ln \widetilde{G}_k dv d\eta_{l_k} - \int f_k \ln f_k dv d\eta_{l_k}.
\end{align}
Here we have equality if and only if $G_k= \widetilde{G}_k$, which means $\Lambda_k= \Theta_k$.
Now, using the first part of lemma \ref{Mtilde} and remark \ref{one}, we can estimate $\int \widetilde{G}_k \ln \widetilde{G}_k dv d\eta_{l_k}$ by $\int f_k \ln f_k dv d\eta_{l_k}$. So, all in all, we obtain $S_k(f_k) \leq 0$ with equality if and only if $f_k=M_k$ and $\Lambda_k=\Theta_k =T_k^{r}= T_k^{t}$.\\
\end{proof}
\subsubsection{H-Theorem for mixtures of polyatomic molecules}
\label{sec6.2}
\textcolor{black}{Define 
$\frac{1}{z_k} := \frac{1}{Z_k^r} \frac{d+l_k}{d}, ~ k=1,2$ and the
 total entropy \begin{align}
 H(f_1,f_2) = \int (f_1 \ln f_1  + 3 z_1 \widehat{G}_1 \ln \widehat{G}_1) dv d\eta_{l_1} + \int (f_2 \ln f_2+3 z_2 \widehat{G}_2 \ln \widehat{G}_2 ) dvd\eta_{l_2}.
 \label{H-funct}
 \end{align}}  \textcolor{black}{
We start with an inequality which is used to prove lemma 3.5. Lemma 3.5 is then needed to prove the H-theorem in theorem 3.5}
\begin{lemma}
Assume $f_1, f_2 >0$.
Assume the relationship between the collision frequencies \eqref{coll} , the conditions for the interspecies Maxwellians \eqref{density}, \eqref{convexvel}, \eqref{veloc}, \eqref{contemp} and \eqref{temp} and the positivity of all temperatures, then
\begin{align}
\varepsilon \frac{d}{2} \ln \Lambda_{12} + \varepsilon \frac{l_1}{2} \ln \Theta_{12} +  \frac{d}{2} \ln \Lambda_{21} +  \frac{l_2}{2} \ln \Theta_{21} \geq \frac{d}{2} \varepsilon \ln \Lambda_1 + \frac{d}{2} \ln \Lambda_2 + \frac{l_1}{2} \varepsilon \ln \Theta_1 + \frac{l_2}{2} \ln \Theta_2.
\end{align}
\label{inequ}
\end{lemma}
\begin{proof}
First we consider the part $E_1:= \frac{d}{2} \ln \Lambda_{12} + \frac{l_1}{2} \ln \Theta_{12}.$ We insert the definitions of $\Lambda_{12}$ and $\Theta_{12}$ into $E_1$ and use the monotonicity of $\ln$ to drop the velocity term. Then we obtain 
$$ E_1 \geq \frac{d}{2} \ln ( \alpha \Lambda_1 + (1- \alpha) \Lambda_2) + \frac{l_1}{2} \ln ( \frac{l_1}{l_1+l_2} \Theta_1 + \frac{l_2}{l_1+l_2} \Theta_2).$$
Now we use that $\ln$ is concave and get
\begin{align}
E_1 \geq \frac{d}{2} \alpha \ln \Lambda_1 + \frac{d}{2} (1- \alpha) \ln \Lambda_2 + \frac{l_1}{2} \frac{l_1}{l_1+l_2} \ln \Theta_1 + \frac{l_1}{2} \frac{l_2}{l_1+l_2} \ln\Theta_2.
\label{E1}
\end{align}
Doing the same with the second part $E_2:= \frac{d}{2} \ln \Lambda_{21} + \frac{l_2}{2} \ln \Theta_{21}$ using that $\frac{l_1}{l_1+l_2} \varepsilon \leq 1$, we obtain
 \begin{align}
 E_2 \geq \frac{d}{2} \varepsilon (1- \alpha) \ln \Lambda_1 + \frac{d}{2} (1- \varepsilon (1 - \alpha)) \ln \Lambda_2 + \frac{l_2}{2}  \varepsilon \frac{l_1}{l_1+l_2} \ln \Theta_1 + \frac{l_2}{2} ( 1-  \varepsilon \frac{l_1}{l_1+l_2}) \ln \Theta_2.
 \label{E2}
 \end{align}
 Multiplying \eqref{E1} by $\varepsilon$ and adding \eqref{E2}, we get
 $$ \varepsilon E_1 + E_2 \geq \frac{d}{2} \varepsilon \ln \Lambda_1 + \frac{d}{2} \ln \Lambda_2 + \frac{l_1}{2} \varepsilon \ln \Theta_1 + \frac{l_2}{2} \ln \Theta_2.$$
 which is the required inequality.
\end{proof}
\begin{lemma}
Assume $f_1, f_2 >0$.
Assume the relationship between the collision frequencies \eqref{coll}, the conditions for the interspecies Maxwellians \eqref{density}, \eqref{convexvel}, \eqref{veloc}, \eqref{contemp} and \eqref{temp} and the positivity of all temperatures. Then
\begin{align*}
\nu_{12} n_2 \int M_{12} \ln M_{12} dv d\eta_{l_1}  + \nu_{21} n_1 \int M_{21} \ln M_{21} dv d\eta_{l_2} \\ \leq \nu_{12} n_2 \int M_1 \ln M_1 dv d\eta_{l_1} + \nu_{21} n_1 \int M_2 \ln M_2 dv d\eta_{l_2}.
\end{align*}
\label{inequ_cor}
\end{lemma}
\begin{proof}
Using that $\ln M_{12} = \ln(\frac{n_2}{\sqrt{2 \pi \frac{\Lambda_{12}}{m_1}}^{d} } \frac{1}{\sqrt{2 \pi\frac{\Theta_{12}}{m_1}}^{l_1}}) - \frac{|v-u_{12}|^2}{2 \frac{\Lambda_{12}}{m_1}} - \frac{|\eta_{l_1}|^2}{2 \frac{\Theta_{12}}{m_1}})$, \\ $\ln M_{21} = \ln(\frac{n_1}{\sqrt{2 \pi \frac{\Lambda_{21}}{m_2}}^{d} } \frac{1}{\sqrt{2 \pi \frac{\Theta_{21}}{m_2}}^{l_2}}) - \frac{|v-u_{21}|^2}{2 \frac{\Lambda_{21}}{m_2}} - \frac{|\eta_{l_2}|^2}{2 \frac{\Theta_{21}}{m_2}})$\\
$\ln M_k = \ln(\frac{n_k}{\sqrt{2 \pi \frac{\Lambda_k}{m_k}}^d} \frac{1}{\sqrt{2 \pi \frac{\Theta_k}{m_k}}^{l_k}}) - \frac{|v-u_k|^2}{2 \frac{\Lambda_k}{m_k}}) - \frac{|\eta_{l_k}|^2}{2 \frac{\Theta_k}{m_k}}$, $k=1,2$, we compute the integrals and obtain that the required inequalities are equivalent to
\begin{align*}
\varepsilon \ln(\frac{n_1}{\sqrt{2 \pi \frac{\Lambda_{12}}{m_1}}^d} \frac{1}{\sqrt{2 \pi \frac{\Theta_{12}}{m_1}}^{l_1}}) +\ln(\frac{n_2}{\sqrt{2 \pi \frac{\Lambda_{21}}{m_2}}^d} \frac{1}{\sqrt{2 \pi \frac{\Theta_{21}}{m_2}}^{l_2}})\\ \leq \varepsilon \ln(\frac{n_1}{\sqrt{2 \pi \frac{\Lambda_{1}}{m_1}}^d} \frac{1}{\sqrt{2 \pi \frac{\Theta_{1}}{m_1}}^{l_1}}) +\ln(\frac{n_2}{\sqrt{2 \pi \frac{\Lambda_{2}}{m_2}}^d} \frac{1}{\sqrt{2 \pi \frac{\Theta_{2}}{m_2}}^{l_2}}).
\end{align*} \textcolor{black}{ This inequality is true since after a brief manipulation of the inequality it is equivalent to lemma \ref{inequ}.}
\end{proof}
\begin{theorem}[H-theorem for mixture]
Assume $f_1, f_2 >0$. Assume 
$\nu_{11} n_1 \geq \nu_{12} n_2$, $\nu_{22} n_2 \geq \nu_{21} n_1$, $\alpha, \delta\neq 1, l_1,l_2 \neq 0.$
Assume the relationship between the collision frequencies \eqref{coll}, the conditions for the interspecies Maxwellians \eqref{density}, \eqref{convexvel}, \eqref{veloc}, \eqref{contemp} and \eqref{temp} and the positivity of all temperatures, then
\begin{align*}
\sum_{k=1}^2 [ \nu_{kk} n_k \int (G_k - f_k) \ln f_k dv d\eta_{l_k} + 
\nu_{kk} n_k
 \int (\widetilde{G}_k - \widehat{G}_k) \ln \widehat{G}_k dv d\eta_{l_k}] \\+ 
 \nu_{11} n_1
  \int (\widetilde{G}_1 - \widehat{G}_1) \ln \widehat{G}_1 dv d\eta_{l_1} + 
  \nu_{22} n_2
  \int (\widetilde{G}_2 - \widehat{G}_2) \ln \widehat{G}_2 dv d\eta_{l_2} \\ +\nu_{12} n_2 \int (M_{12} - f_1) \ln f_1 dv d\eta_{l_1} + \nu_{21} n_1 \int (M_{21} - f_2) \ln f_2 dv d\eta_{l_2}  \leq 0,
\end{align*}
with equality if and only if $f_1$ and $f_2$ are in local equilibrium (see definition \ref{localequ}).
\label{H-theorem}
\end{theorem}
\begin{remark}
The inequality in the H-Theorem is still true if $l_1=0$ or $l_2=0$ which means that one species is mono atomic. In this case only the equalities with $\Theta_1$ and $\Theta_2$, respectively in the local equilibrium vanish.
\end{remark}
\begin{proof}
The fact that $ \nu_{kk} n_k \int (G_k - f_k) \ln f_k dv d\eta_{l_k} + 
\nu_{kk} n_k
\int (\widetilde{G}_k - \widehat{G}_k) \ln \widehat{G}_k dv d\eta_{l_k} \leq 0, k=1,2$ is shown in Lemma \ref{one_species}.
In both cases we have equality if and only if $f_1=G_1$ with $\Lambda_1=\Theta_1=T_1^{t}=T_1^{r}$ and $f_2 = G_2$ with $\Lambda_2=\Theta_2=T_2^{t}=T_2^{r}$. \\
Let us define 
\begin{align*}
S(f_1,f_2) :=
\nu_{11} n_1
 \int (\widetilde{G}_1 - \widehat{G}_1) \ln \widehat{G}_1 dv d\eta_{l_1} + 
 \nu_{22} n_2
 \int (\widetilde{G}_2 - \widehat{G}_2) \ln \widehat{G}_2 dv d\eta_{l_2} \\ +\nu_{12} n_2 \int (M_{12} - f_1) \ln f_1 dv d\eta_{l_1} + \nu_{21} n_1 \int (M_{21} - f_2) \ln f_2 dv d\eta_{l_2}. 
\end{align*}
The task is to prove that $S(f_1, f_2)\leq 0$.
Since the function $H(x)= x \ln x -x$ is strictly convex for $x>0$, we have $H'(f) (g-f) \leq H(g) - H(f)$ with equality if and only if $g=f$. So \begin{align}
(g-f) \ln f  \leq g \ln g - f \ln f +f -g. 
\label{convex}
\end{align}
Consider now $S(f_1,f_2)$ and apply the inequality \eqref{convex} to each of the terms in $S$.
\begin{align*}
S \leq
\nu_{11} n_1
[ \int \widetilde{G}_1 \ln \widetilde{G}_1  dv d\eta_{l_1} - \int \widehat{G}_1 \ln \widehat{G}_1 dv d\eta_{l_1}  + \int \widehat{G}_1 dv d\eta_{l_1} - \int \widetilde{G}_1 dv d\eta_{l_1} ]\\ + \nu_{12} n_2 [ \int M_{12} \ln M_{12} dv d\eta_{l_1} - \int f_1 \ln f_1 dv d\eta_{l_1}+ \int f_1 dv\eta - \int M_{12} dv d\eta_{l_1} ]\\+ \nu_{21} n_1 [ \int M_{21} \ln M_{21} dv d\eta_{l_2} - \int f_2 \ln f_2 dv d\eta_{l_2} + \int f_2 dv d\eta - \int M_{21} dv d\eta_{l_2} ]\\  + 
\nu_{22} n_2
 [ \int \widetilde{G}_2 \ln \widetilde{G}_2 dv d\eta_{l_2} - \int \widehat{G}_2 \ln \widehat{G}_2 dv d\eta_{l_2} + \int \widehat{G}_2 dv d\eta_{l_2} - \int \widetilde{G}_2 dv d\eta_{l_2}],
\end{align*}
with equality if and only if $f_1=M_{12}$, $f_2=M_{21}$, $\widetilde{G}_1=\widehat{G}_1$ and $\widetilde{G}_2=\widehat{G}_2$. Combining this with the condition for equality of the single collision term $f_1=G_1$ with $\Lambda_1=\Theta_1=T_1^{t}=T_1^{r}$ and $f_2=G_2$ with $\Lambda_2=\Theta_2=T_2^{t}=T_2^{r}$, we get that we have equality if and only if we are in local equilibrium (see definition \ref{localequ}). 
Since $\widehat{G}_1, \widetilde{G}_1, f_1$ and $M_{12}$ have the same density and $\widehat{G}_2, \widetilde{G}_2, M_{21}$ and $f_2$ have the same density, too,  the right-hand side reduces to
\begin{align*}
S \leq
\nu_{11} n_1
[ \int \widetilde{G}_1 \ln \widetilde{G}_1  dv d\eta_{l_1} - \int \widehat{G}_1 \ln \widehat{G}_1 dv d\eta_{l_1} ]\\ + \nu_{12} n_2 [ \int M_{12} \ln M_{12} dv d\eta_{l_1} - \int f_1 \ln f_1 dv d\eta_{l_1} ]\\+ \nu_{21} n_1 [ \int M_{21} \ln M_{21} dv d\eta_{l_2} - \int f_2 \ln f_2 dv d\eta_{l_2} ]\\  + 
\nu_{22} n_2
[ \int \widetilde{G}_2 \ln \widetilde{G}_2 dv d\eta_{l_2} - \int \widehat{G}_2 \ln \widehat{G}_2 dv d\eta_{l_2}].
\end{align*}
According to the second part of lemma \ref{Mtilde}, 
we obtain
\begin{align*}
S \leq
\nu_{11} n_1 [ \int \widetilde{G}_1 \ln \widetilde{G}_1  dv d\eta_{l_1} - \int G_1 \ln G_1 dv d\eta_{l_1} ]\\ + \nu_{12} n_2 [ \int M_{12} \ln M_{12} dv d\eta_{l_1} - \int f_1 \ln f_1 dv d\eta_{l_1} ]\\+ \nu_{21} n_1 [ \int M_{21} \ln M_{21} dv d\eta_{l_2} - \int f_2 \ln f_2 dv d\eta_{l_2} ]\\  + \nu_{22} n_2 [ \int \widetilde{G}_2 \ln \widetilde{G}_2 dv d\eta_{l_2} - \int G_2 \ln G_2 dv d\eta_{l_2}].
\end{align*}
According to lemma \ref{inequ_cor}, the last part of lemma \ref{Mtilde} and the assumption that $\nu_{kk} n_k \geq \nu_{kj} n_j, ~ k,j=1,2, k \neq j$, we get
\begin{align*}
S \leq
\nu_{11} n_1 [ \int \widetilde{G}_1 \ln \widetilde{G}_1  dv d\eta_{l_1} - \int G_1 \ln G_1 dv d\eta_{l_1} ]\\ + \nu_{12} n_2 [ \int G_{1} \ln G_{1} dv d\eta_{l_1} - \int f_1 \ln f_1 dv d\eta_{l_1} ]\\+ \nu_{21} n_1 [ \int G_{2} \ln G_{2} dv d\eta_{l_2} - \int f_2 \ln f_2 dv d\eta_{l_2} ]\\ + \nu_{22} n_2 [ \int \widetilde{G}_2 \ln \widetilde{G}_2 dv d\eta_{l_2} - \int G_2 \ln G_2 dv d\eta_{l_2}]\\ \leq  \nu_{12} n_2 [ \int \widetilde{G}_1 \ln \widetilde{G}_1  dv d\eta_{l_1}  - \int f_1 \ln f_1 dv d\eta_{l_1} ]\\+ \nu_{21} n_1 [ \int \widetilde{G}_2 \ln \widetilde{G}_2 dv d\eta_{l_2}   - \int f_2 \ln f_2 dv d\eta_{l_2} ].
\end{align*}
which leads to $S\leq 0$ using the \textcolor{black}{first} part of lemma \ref{Mtilde} and remark \ref{one}.

\end{proof}
 Using the definition \eqref{H-funct}, we can compute 
\begin{align*} \partial_t H(f_1,f_2) + \nabla_x \cdot \int (f_1 \ln f_1  + 2 z_1 \widehat{G}_1 \ln \widehat{G}_1)v dv d\eta_{l_1}\\+ \nabla_x \cdot\int (f_2 \ln f_2 + 2 z_2 \widehat{G}_2 \ln \widehat{G}_2 ) v dvd\eta_{l_2} = S(f_1,f_2) \textcolor{black}{+R(f_1,f_2)},
\end{align*}
 by multiplying the BGK equation for species $1$ by $\ln f_1$, the BGK equation for the species $2$ by $\ln f_2$, equations \eqref{kin_TempES} by $3 z_k\ln G_k$ and sum the integrals with respect to $v$ and $\eta_{l_1}$ and $\eta_{l_2}$, respectively. \textcolor{black}{The remaining term $R(f_1,f_2)$  can be bounded by zero from below by an explicit computation  assuming that $\Lambda_k$ and $\Theta_k$ are bounded from below and above and assume that $T_k^{rot} \geq \tilde{C} \Theta_k$ for an appropriate $\tilde{C}$ and $z_k$ small enough.  The additional estimate $T_k^{rot} \geq \tilde{C} \Theta_k$ helps to indicate how to choose the initial data of the artificial temperature $\Theta_k$.}
 \begin{cor}[Entropy inequality for mixtures]
Assume $f_1, f_2 >0$, \textcolor{black}{ $\Lambda_k$ and $\Theta_k$ are bounded from below and above and $T_k^{rot} \geq \tilde{C} \Theta_k$ for an appropriate $\tilde{C}$ and $z_k$ small enough.}
Assume relationship \eqref{coll}, the conditions \eqref{density}, \eqref{convexvel}, \eqref{veloc}, \eqref{contemp} and \eqref{temp} and the positivity of all temperatures \eqref{gamma}, then we have the following entropy inequality
\begin{align*}
\partial_t \left(  H(f_1,f_2) \right) \\+ \nabla_x \cdot \left(\int  v (f_1 \ln f_1  + 3 z_1 \widehat{G} \ln \widehat{G}_1) dv d\eta_{l_1} + \int \textcolor{black}{v}( f_2 \ln f_2+ 3 z_2 \widehat{G}_2 \ln \widehat{G}_2 ) dv d\eta_{l_2} \right) \leq 0,
\end{align*}
with equality if and only if $f_1$ and $f_2$ are in local equilibrium (see definition \ref{localequ}).
\end{cor}
\begin{remark}
By computing the integrals 
\textcolor{black}{ \begin{align*} \int \widehat{G}_k \ln \widehat{G}_k dv d\eta_{l_k}  \quad \text{and} \quad \int v \widehat{G}_k \ln \widehat{G}_k dv d\eta_{l_k}\quad \text{for} \quad k=1,2,
 \end{align*}}
   we see that \textcolor{black}{{\small $$\partial_t [\int \widehat{G}_1 \ln \widehat{G}_1 dv d\eta_{l_1}+\int \widehat{G}_2 \ln \widehat{G}_2 dv d\eta_{l_2}] + \nabla_x \cdot [\int v \widehat{G}_1 \ln \widehat{G}_1 dv d\eta_{l_1}+\int v \widehat{G}_2 \ln \widehat{G}_2 dv d\eta_{l_2}] \leq 0,$$}} is equivalent to 
\textcolor{black}{   $$\partial_t (\det(\Lambda_1^{ten}) \Theta_1^{l_1}+\det(\Lambda_2^{ten}) \Theta_2^{l_2}) + \nabla_x \cdot ((\det(\Lambda_1^{ten}) \Theta_1^{l_1}+\det(\Lambda_2^{ten}) \Theta_2^{l_2}) u_k) \leq 0,$$ }
   so we could also consider the entropy 
\textcolor{black}{   $$H(f_1,f_2)=\sum_{k=1}^2 \int f_k \ln f_k dv d\eta_{l_k} + z_1 \det(\Lambda_1^{ten}) \Theta_1^{l_1}+z_2 \det(\Lambda_2^{ten}) \Theta_2^{l_2}.$$}
   
\end{remark}
\section{Comparison with the ES-BGK model for one species of polyatomic molecules by Andries, Le Tallec, Perlat and Perthame}
\label{sec7}
We will now consider a different ES-BGK model for a single species ES-BGK model of polyatomic molecules.
In \cite{Perthame}, they consider a distribution function $f(t,x,v,I)$ depending on the position $x\in \mathbb{R}^3$, the velocity $v\in\mathbb{R}^3$ and internal energy $\varepsilon(I)= I^{\frac{2}{\delta}}$, $I\in\mathbb{R}^+$ at time $t$. $\delta$ denotes the number of degrees of freedom in internal energy. \textcolor{black}{In \cite{Perthame}, it is assumed that the mass of the particles is equal to $1$. In the following, we assume additionally that $k_B=1$ in this model.} The mass density $\rho$ and mean velocity $u$ are defined as in the model described in the previous subsection integrating with respect to $v$ and $I$. The energy is defined as 
$$ E(x,t) = \int \int (\frac{1}{2} |v|^2 + I^{\frac{\delta}{2}}) f dv dI = \frac{1}{2} \rho |u|^2 + \rho e.$$
The specific internal energy can be divided into 
$$ e_{tr} = \frac{1}{\rho} \int \int \frac{1}{2} |v-u|^2 f dv dI,$$
$$e_{int} = \frac{1}{\rho} \int \int  I^{\frac{2}{\delta}} f dv dI,$$
and associate with this the corresponding temperatures
$$e= e_{tr} + e_{int}= \frac{3+ \delta}{2} R T_{equ},$$
$$ e_{tr}=\frac{3}{2} R T_{tr},$$
$$e_{int} = \frac{\delta}{2} R T_{int},$$
and define $T_{rel} = \theta T_{equ} + (1- \theta) T_{int}.$ They consider the generalized Gaussian for the single species ES-BGK model
$$ \widetilde{G}[f]= \frac{\rho \Lambda_{\delta}}{\sqrt{\det(2 \pi \mathcal{T})}} \frac{1}{R T_{rel}^{\frac{\delta}{2}}} \exp( - \frac{1}{2} (v-u) \cdot \mathcal{T}^{-1} \cdot(v-u) + \frac{I^{\frac{\delta}{2}}}{R T_{rel}}),$$ 
with the tensor $\mathcal{T} = (1- \theta) ((1- \nu) R T_{tr} \textbf{1} + \nu \Theta) + \theta R T_{equ} \textbf{1}$ where only the translational part is replaced by a tensor. $\Theta$ denotes the pressure tensor, $\Lambda_{\delta}$ is a constant ensuring that the integral of $\widetilde{G}[f]$ with respect to $v$ and $I$ is equal to the density $\rho$ and $R$ is the gas constant.
The convex combination in $\theta$ takes into account that $T_{tr}$ and $T_{int}$ relaxes towards the common value $T_{equ}$. In the space-homogeneous case we see that we get the following macroscopic equations
\begin{align*}
\partial_t T_{tr} &= C  ( T_{tr} (1- \theta) + \theta T_{equ} - T_{tr})= C \theta (T_{equ} - T_{tr}),
\\
\partial_t T_{int} &= C \theta (T_{equ} - T_{int}), 
\end{align*}
with some coefficient $C$. These macroscopic equations describe a relaxation of $T_{tr}$ and $T_{int}$ towards $T_{equ}$.

In this paper, we took \cite{Bernard} as basis to extend it to mixtures. 
The main differences of the model in \cite{Perthame} and the model in \cite{Bernard} are the following. The model in \cite{Perthame} has one variable $I\in \mathbb{R}^+$ for all degrees of freedom in internal energy and the model in \cite{Bernard} has one variable $\eta \in \mathbb{R}^M$ to each degree of freedom in internal energy. Moreover, the relaxation of the translational and rotational/vibrational temperatures to a common value is done in \cite{Perthame} by introducing a relaxation temperature $T_{rel}$ and in the model \cite{Bernard} it is done by the additional relaxation equation \eqref{kin_TempES}.

\section{Applications}
\subsection{Chu reduction}
\label{sec8}
In order to reduce the complexity of the variable for rotational and vibrational energy  degrees of freedom $\mu_1,....\mu_{l_k}$ we apply the Chu reduction proposed in \cite{Chu}. It is possible to apply the Chu reduction since $\eta_1,...\eta_{l_k}$ do not appear in the transport operators in \eqref{ESBGK}.
We consider the system of equations
\begin{align}
\begin{split}
\partial_t f_1 + v \cdot \nabla_{x} f_1 = \nu_{11} n_1 (G_1 - f_1) + \nu_{12} n_2 ( M_{12} - f_1), \\
\partial_t f_2 + v \cdot \nabla_{x} f_2 = \nu_{22} n_2 (G_2 - f_2) + \nu_{21} n_1 ( M_{21} - f_2).
\end{split}
\end{align}
Now, consider the reduced functions
$$ g_1 =  \int f_1 d\eta_{l_1}, \quad g_2 =  \int f_2 d\eta_{l_2}.$$
Then they satisfy the equations
\begin{align}
\begin{split}
\partial_t g_1 + v \cdot \nabla_{x} g_1 = \nu_{11} n_1 (\widetilde{G}_1 - g_1) + \nu_{12} n_2 ( \widetilde{M}_{12} - g_1), \\
\partial_t g_2 + v \cdot \nabla_{x} g_2 = \nu_{22} n_2 (\widetilde{G}_2 - g_2) + \nu_{21} n_1 ( \widetilde{M}_{21} - g_2),
\end{split}
\end{align}
where $\widetilde{G}_1, \widetilde{G}_2, \widetilde{M}_{12}$ and $\widetilde{M}_{21}$ are given by
$$\widetilde{G}_1 = \int  G_1 d\eta_{l_1}, \quad \widetilde{M}_{12} =  \int M_{12} d\eta_{l_1}$$
$$\widetilde{G}_2 = \int  G_2 d\eta_{l_2}, \quad \widetilde{M}_{21} =  \int M_{21} d\eta_{l_2}.$$
It is possible to compute the densities
$$ n_1 = \int \int f_1 d\eta_{l_1} dv =  \int g_1 dv, $$
$$ n_2 = \int \int \int f_2 d\eta_{l_2} dv = \int g_2 dv,$$
the velocities 
$$ u_{1}= \int \int \int v f_1 d\eta_{l_1} dv = \int v g_1 dv,$$
$$ u_{2}= \int \int \int v f_2 d\eta_{l_2} dv = \int v g_2 dv,$$
the temperatures
\begin{align*}
 \Lambda_1 &= \frac{1}{n_1} \int \int |v-u_1|^2 f_1 d\eta_{l_1} dv = \frac{1}{n_1} \int |v-u_1|^2 g_1 dv,\\ \Lambda_2 &= \frac{1}{n_2} \int |v-u_2|^2 g_2 dv,\\ \Theta_1 &= \frac{1}{n_1} \int \int |\eta_{l_1}|^2 f_1 d\eta_{l_1} dv = \frac{1}{n_1} \int |\eta_{l_1}|^2 h_1 dv\\ \Theta_2 &= \frac{1}{n_2} \int |\eta_{l_2}|^2 h_2 dv,
 \end{align*}
if we define the reduced functions
$$ h_1 = \int  |\eta_{l_1}|^2 f_1 d\eta_{l_1}, \quad h_2 = \int  |\eta_{l_2}|^2 f_2 d\eta_{l_2},$$
which solve the equations
\begin{align}
\begin{split}
\partial_t h_1 + v \cdot \nabla_{x} h_1 = \nu_{11} n_1 (\widetilde{\widetilde{G}}_1 - h_1) + \nu_{12} n_2 ( \widetilde{\widetilde{M}}_{12} - h_1), \\
\partial_t h_2 + v \cdot \nabla_{x} h_2 = \nu_{22} n_2 (\widetilde{\widetilde{G}}_2 - h_2) + \nu_{21} n_1 ( \widetilde{\widetilde{M}}_{21} - h_2),
\end{split}
\end{align}
where $\widetilde{\widetilde{G}}_1, \widetilde{\widetilde{G}}_2, \widetilde{\widetilde{M}}_{12}$ and $\widetilde{\widetilde{M}}_{21}$ are given by
$$\widetilde{\widetilde{G}}_1 = \int  |\eta_{l_1}|^2 G_1 d\eta_{l_1}, \quad \widetilde{\widetilde{M}}_{12} = \int  |\eta_{l_1}|^2  M_{12} d\eta,$$
$$\widetilde{\widetilde{G}}_2 = \int  |\eta_{l_2}|^2  G_2 d\eta_{l_2}, \quad \widetilde{\widetilde{M}}_{21} = \int  |\eta_{l_2}|^2  M_{21} d\eta_{l_2}.$$
If we compute $\widetilde{G}_k$, $\widetilde{M}_{12}$, $\widetilde{M}_{21}$, $\widetilde{\widetilde{G}}_k$, $\widetilde{M}_{12}$, $\widetilde{M}_{21}$ for $k=1,2$, we get
\begin{align} 
\begin{split}
\widetilde{G}_k(x,v,t) &= \frac{n_k}{\sqrt{ \det (2 \pi\frac{\Lambda_k^{ES}}{m_k})} }  \exp(-m_k (v-u_k) \cdot (\Lambda_k^{ES})^{-1} \cdot (v-u_k)  ), \quad k=1,2,
\\
\widetilde{M}_{12}(x,v,t) &= \frac{n_{1}}{\sqrt{2 \pi \frac{\Lambda_{12}}{m_1}}^d }  \exp({- \frac{|v-u_{12}|^2}{2 \frac{\Lambda_{12}}{m_1}}}),
\\
\widetilde{M}_{21}(x,v,t) &= \frac{n_{2}}{\sqrt{2 \pi \frac{\Lambda_{21}}{m_2}}^d } \exp({- \frac{|v-u_{21}|^2}{2 \frac{\Lambda_{21}}{m_2}}}),
\end{split}
\end{align}
\begin{align} 
\begin{split}
\widetilde{\widetilde{G}}_k(x,v,t) &= \frac{n_k}{\sqrt{ \det (2 \pi\frac{\Lambda_k^{ES}}{m_k})} }  \exp(-m_k (v-u_k) \cdot (\Lambda_k^{ES})^{-1} \cdot (v-u_k)  ) \Theta_k, \quad k=1,2,
\\
\widetilde{\widetilde{M}}_{12}(x,v,t) &= \frac{n_{1}}{\sqrt{2 \pi \frac{\Lambda_{12}}{m_1}}^d }  \exp({- \frac{|v-u_{12}|^2}{2 \frac{\Lambda_{12}}{m_1}}}) \Theta_{12},
\\
\widetilde{\widetilde{M}}_{21}(x,v,t) &= \frac{n_{2}}{\sqrt{2 \pi \frac{\Lambda_{21}}{m_2}}^d } \exp({- \frac{|v-u_{21}|^2}{2 \frac{\Lambda_{21}}{m_2}}}) \Theta_{21},
\end{split}
\end{align}
We are able to compute all the six Maxwell distributions because we can compute all moments by the previous computation.
\subsection{A mixture consisting of a mono and a diatomic gas}
\label{sec9}
We consider now the special case of two species, one species is mono-atomic and has only translational degrees of freedom $l_1=0$, the other one is diatomic and has in addition two rotational degrees of freedom $l_2=2$ and both have the number of degrees of freedom in translations given by $d$ with $d\in \mathbb{N}$. In this case the total number of rotational degrees of freedom is $M=l_1+l_2=2$ since in sum we have two possible rotations. Our variables for the rotational energy degrees of freedom are $\eta \in \mathbb{R}^2$, $ \eta_{l_1} = \begin{pmatrix}
0 \\ 0 
\end{pmatrix},$ $ \eta_{l_2}= \eta $, since $\eta_{l_k}$ coincides with $\eta$ in the components corresponding to the rotational degrees of freedom of species $k$ and is zero in the other components. 
So our distribution function $f_1(x,v,t)$ of species $1$ depends on $x,v,$ and $t$ and our distribution function $f_2(x,v,\eta,t)$ of species $2$ depends on $x,v, \eta$ and $t$. The moments of $f_1$ are given by 
\begin{align}
\int f_1(v) \begin{pmatrix}
1 \\ v  \\ m_1 |v-u_1|^2  \\ m_1 (v-u_1(x,t)) \otimes (v-u_1(x,t))
\end{pmatrix} 
dv =: \begin{pmatrix}
n_1 \\ n_1 u_1  \\ d n_1 T_1^{t}  \\ \mathbb{P}_1,
\end{pmatrix} 
\end{align} 
and the moments of species $2$ are given by 
\begin{align}
\int f_2(v, \eta) \begin{pmatrix}
1 \\ v \\ \eta \\ m_2 |v-u_2|^2 \\ m_2 |\eta  |^2 \\ m_2 (v-u_2(x,t)) \otimes (v-u_2(x,t))
\end{pmatrix} 
dvd\eta=: \begin{pmatrix}
n_2 \\ n_2 u_2 \\ 0 \\ d n_2 T_2^{t} \\ l_2 n_2 T_k^{r} \\ \mathbb{P}_2
\end{pmatrix} .
\end{align} 
The third equality is an assumption. We could also consider a general $\bar{\eta}$.
Our model reduces to 
\begin{align} \begin{split} 
\partial_t f_1 + \nabla_x \cdot (v f_1)   &= \nu_{11} n_1 (G_1(f_1) - f_1) + \nu_{12} n_2 (M_{12}(f_1,f_2)- f_1),
\\ 
\partial_t f_2 + \nabla_x \cdot (v f_2) &=\nu_{22} n_2 (G_2(f_2) - f_2) + \nu_{21} n_1 (M_{21}(f_1,f_2)- f_2), 
\end{split}
\end{align}
with the modified Maxwellians
{\footnotesize
\begin{align} 
\begin{split}
G_1(f_1)(x,v,t) &= \frac{n_1}{\sqrt{\det(2 \pi \frac{\Lambda^{ES}_1}{m_1})} }   \exp\left({- \frac{1}{2} (v-u_1) \cdot \left(\frac{\Lambda_1^{ES}}{m_1}\right)^{-1} \cdot (v-u_1)}\right), 
\\
G_2(f_2)(x,v,\eta,t) &= \frac{n_2}{\sqrt{\det(2 \pi \frac{\Lambda^{ES}_2}{m_2})} } \frac{1}{\sqrt{2 \pi \frac{\Theta_2}{m_2}}^{l_2}}  \exp\left({- \frac{1}{2} (v-u_2) \cdot \left(\frac{\Lambda_2^{ES}}{m_2}\right)^{-1} \cdot (v-u_2)}- \frac{1}{2} \frac{m_2 |\eta|^2}{\Theta_2}\right), 
\\
M_{12}(x,v,t) &= \frac{n_{12}}{\sqrt{2 \pi \frac{\Lambda_{12}}{m_1}}^d }  \exp \left({- \frac{|v-u_{12}|^2}{2 \frac{\Lambda_{12}}{m_1}}}\right),
\\
M_{21}(x,v,\eta,t) &= \frac{n_{21}}{\sqrt{2 \pi \frac{\Lambda_{21}}{m_2}}^d } \frac{1}{\sqrt{2 \pi \frac{\Theta_{21}}{m_2}}^{l_2}} \exp \left({- \frac{|v-u_{21}|^2}{2 \frac{\Lambda_{21}}{m_2}}}- \frac{|\eta|^2}{2 \frac{\Theta_{21}}{m_2}}\right),
\end{split}
\end{align}}
where
\begin{align*}
\Lambda_1^{ES} &= (1- \mu_1) T_1^{t} \textbf{1}_n + \mu_1 \frac{\mathbb{P}_1}{n_1}, \\
\Lambda_2^{ES} &= (1- \mu_2) \Lambda_2 \textbf{1}_n + \mu_2 \frac{\Lambda_2^{ten}}{n_2}, \\
\end{align*}
 with $\mu_k \in \mathbb{R}$, $k=1,2$. 
 For $\Lambda_2^{ten}$ we use the additional relaxation equation 
 \begin{align}
\partial_t \widehat{G}_2 + v \cdot \nabla_x \widehat{G}_2 = \frac{\nu_{22} n_2}{Z_r^2} \frac{d+2}{d} (\widetilde{G}_2 - \widehat{G}_2) + \nu_{22} n_2 (G_2-f_2) + \nu_{21} n_1 (M_{21} - f_2), 
\label{kin_TempESapp}
\end{align}
Here $\widehat{G}_2$ is given by 
\begin{align}
\widehat{G}_2= \frac{n_2}{\sqrt{\det( 2 \pi \frac{\Lambda_2^{ten}}{m_2})}} \exp \left(- \frac{1}{2} (v-u_2) \cdot \left(\frac{\Lambda_2^{ten}}{m_2}\right)^{-1}\cdot (v- u_2)- \frac{m_2|\eta|^2}{2 T_2^{r}} \right), \quad k=1,2.
\label{Max_equapp}
\end{align}
and $\widetilde{G}_2$ is given by
\begin{align}
\widetilde{G}_2 = \frac{n_2}{\sqrt{\det(2\pi \frac{T_2^{ten}}{m_2}})} \frac{1}{\sqrt{2 \pi \frac{T_2}{m_2}}^{2}} \exp \left( - \frac{1}{2}(v-u_2) \cdot \left( \frac{T_2^{ten}}{m_2}\right)^{-1} \cdot (v-u_2) - \frac{1}{2}  \frac{m_2 |\eta|^2 }{ T_2} \right),
\end{align}
where the components of $T_2^{ten}$ are defined in the following way. 
\begin{align}
\begin{split}
(T_2^{ten})_{ii} &= T_2 := \frac{d}{d+2} \Lambda_2 + \frac{2}{d+2} \Theta_2 \hspace{2.6cm}\text{for} \quad i=1,\dots d, \\
(T_2^{ten})_{ij} &= \frac{d}{d+2} (\mathbb{P}_2)_{ij} \hspace{4.5cm} \text{for} \quad i,j = 1, \dots d, i \neq j,
\end{split}
\label{Tenapp}
\end{align}
 We couple this with conservation of internal energy of species $2$
 \begin{align}
\frac{d}{2} n_2 \Lambda_2 = \frac{d}{2} n_2 T_2^{t} +\frac{l_2}{2} n_2 T_2^{r} - \frac{l_2}{2} n_2 \Theta_2. 
\end{align} 
If we multiply \eqref{kin_TempESapp} by $|\eta|^2$ and integrate with respect to $v$ and $\eta$, this leads to the following macroscopic equation 
\begin{align}
\begin{split}
\partial_t ( \Lambda_2^{ten}) + u_2 \cdot \nabla_x ( \Lambda_2^{ten}) = \frac{\nu_{22} n_2}{Z_r^2} \frac{d+2}{d} ( T_2^{ten} - \Lambda_2^{ten}) &+ \nu_{22} n_2 (\Lambda_2^{ES} - \mathbb{P}_2) \\&+ \nu_{21} n_1 (\Theta_{12} - T_2^{r}).
\end{split}
\end{align}
If we assume that \begin{align*} n_{12}=n_1 \quad \text{and} \quad n_{21}=n_2, \\ u_{12}= \delta u_1 + (1- \delta) u_2, \quad \delta \in \mathbb{R},
\end{align*}
and
\begin{align}
\begin{split}
\Lambda_{12} &=  \alpha T_1^{t} + ( 1 - \alpha) \Lambda_2 + \gamma |u_1 - u_2 | ^2,  \quad 0 \leq \alpha \leq 1, \gamma \geq 0, 
\end{split}
\end{align}

we have conservation of mass, total momentum and total energy provided that
\begin{align}
u_{21}=u_2 - \frac{m_1}{m_2} \varepsilon (1- \delta ) (u_2 - u_1),
\end{align}
\begin{align}
\begin{split}
\Lambda_{21} + \frac{l_2}{d} \Theta_{21}=\left[ \frac{1}{d} \varepsilon m_1 (1- \delta) \left( \frac{m_1}{m_2} \varepsilon ( \delta - 1) + \delta +1 \right) - \varepsilon \gamma \right] |u_1 - u_2|^2 \\+ \varepsilon ( 1 - \alpha ) T_1^{t} + ( 1- \varepsilon ( 1 - \alpha)) \Lambda_2 + \frac{l_2}{d}  \Theta_2.
\end{split}
\end{align}
We take into account the symmetry of the  temperatures and choose
\begin{align}
\begin{split}
\Lambda_{21} &= \varepsilon (1- \alpha) \Lambda_1 + (1- \varepsilon(1-\alpha)) \Lambda_2 \\ &+ \left[ \frac{1}{d} \varepsilon m_1 (1- \delta) \left( \frac{m_1}{m_2} \varepsilon ( \delta - 1) + \delta +1 \right) - \varepsilon \gamma \right] |u_1 - u_2|^2, \end{split} \\
\Theta_{21} &=\Theta_2.
\end{align}


\newpage
\begin{thebibliography}{99}
\bibitem{Aliat} A. Aliat, A. Chikkaoui, E.V. Kustova, \emph{Nonequilibrium kinetics of a radiative CO flow behind a shock wave}, Physical review E68, 056306, 2003

\bibitem{AndriesPerthame2001}
P.Andries, B.Perthame, \emph{The ES-BGK model equation with correct Prandtl number}, AIP conference proceedings, {\bf 30}
30-36, 2001


\bibitem{Perthame}
P.Andries, P: Le Tallec, J. Perlat, B:Perthame, \emph{The Gaussian -BGK model of Boltzmann equation with small Prandtl number,} Eur. J. Mech. B - Fluids, {\bf 19}, 813-830, 2000



\bibitem{Bennoune_2008}
M. Bennoune, M. Lemou and L. Mieussens, 
{\sl Uniformly stable numerical schemes for the Boltzmann equation preserving the compressible Navier-Stokes asymptotics,} Journal of Computational Physics, {\bf 227}, 3781-3803, 2008

\bibitem{Bernard_2015}
F. Bernard, A. Iollo and G. Puppo, 
{\sl Accurate asymptotic preserving boundary conditions for kinetic equations on Cartesian grids,} Journal of Scientific Computing, {\bf 65}, 735-766, 2015

 \bibitem{Bernard} F.Bernard, A.Iollo, G. Puppo, \emph{Polyatomic Models for Rarefied Flows,} submitted, 2017
 

\bibitem{Bisi}
M.Bisi, M. C\'aceres, \emph{A BGK relaxation model for polyatomic gas mixtures,} Communication in Mathematical Sciences, {\bf 14}, 297-325, 2016



\bibitem{brullschneider} 
S.Brull, J.Schneider, \emph{On the ellipsoidal statistical model for polyatomic gases, Continuum Mechanics and Thermodynamics,} {\bf 20}, 489-508, 2009

\bibitem{Chu} C. K. Chu, \emph{Kinetic-theoretic description of the formation of a shock wave.} Phys. Fluids
8, 12?22, 1965

  \bibitem{Cercignani}
 C. Cercignani, {\sl Rarefied Gas Dynamics, From Basic Concepts to Actual Calculations}, Cambridge University Press, 2000 
 
  \bibitem{Cercignani_1975}
 C. Cercignani, {\sl The Boltzmann equation and its applications}, Springer, 1975
 
 \bibitem{Crestetto_2012}
A. Crestetto, N. Crouseilles and M. Lemou,
{\sl Kinetic/fluid micro-macro numerical schemes for Vlasov-Poisson-BGK equation using particles, } Kinetic and Related Models, {\bf 5}, 787-816, 2012
 

\bibitem{Dimarco_2014}
G. Dimarco and L. Pareschi, 
{\sl Numerical methods for kinetic equations,} 
Acta Numerica, {\bf 23}, 369-520, 2014

\bibitem{Jin_2010}
F. Filbet and S. Jin,
{\sl A class of asymptotic-preserving schemes for kinetic equations and related problems with stiff sources,} Journal of Computational Physics, {\bf 20}, 7625-7648, 2010



 


\bibitem{Holway}
L.H.Holway, \emph{ New statistical models for kinetic theory: methods of construction,} The physics of fluids, {\bf 9}, 1658-1673, 1966

\bibitem{Kelly} 
J.Kelly, \emph{ Semiclassical Statistically Mechanics,} lecture notes, 2002



\bibitem{Orlach} J.-M. Orlac'h, V. Giovangigli, T. Novikova, P.R. i Cabarrocas, \emph{Kinetic theory of two-temperature polyatomic plasmas,} Physica A: Statistical Mechanics and its Applications, {\bf 494}, 503-546, 2018

\bibitem{Pirner3} C. Klingenberg, M. Pirner, \emph{Existence, Uniqueness and Positivity of solutions for BGK models for mixtures,}  Journal of Differential equations, 2017

 \bibitem{Pirner} C. Klingenberg, M.Pirner, G.Puppo, \emph{A consistent kinetic model for a two-component mixture with an application to plasma}, Kinetic and related Models, {\bf 10}, 445-465, 2017 
 
 \bibitem{Pirner2} C.Klingenberg, M.Pirner, G.Puppo, \emph{Kinetic ES-BGK models for a multicomponent gas mixture}, Springer Proceedings in Mathematics and Statistics of the International Conference on Hyperbolic Problems: Theory, Numeric and Applications in Aachen 2016, 2017 
 

\bibitem{Morse}
T.F.Morse, \emph{ Kinetic Model for Gases with Internal Degrees of Freedom,} The Physics of Fluids, {\bf}, 159-169, 1964

 
  \bibitem{Puppo_2007}
S. Pieraccini and G. Puppo, 
{\sl Implicit-explicit schemes for BGK kinetic equations}, Journal of Scientific Computing, {\bf 32}, 1-28, 2007


\bibitem{Rykov}
V.A.Rykov, \emph{ A model kinetic equation for a gas with rotational degrees of freedom,} Fluid Dynamics, {\bf 10}, 959-966, 1975





 




 


 








 

 
 
 
 
\end{thebibliography}

\end{document}